%% file: HBRBF-Arxiv.tex
\documentclass[a4paper]{amsart}

\usepackage{graphicx} 
\usepackage{amsmath}
\usepackage{mathptmx}
\usepackage{ams} 
\usepackage{psfrag} 
\usepackage{graphicx}
\usepackage[norelsize]{algorithm2e}
\usepackage[marginratio=1:1,height=600pt,width=400pt,tmargin=117pt]{geometry}
\newcommand{\mc}[1]{\mathcal{#1}}
\newcommand{\mco}{\mathcal{O}}
\newcommand{\mcp}{\mathcal{P}}
\newcommand{\mck}{\mathcal{K}}
\newcommand{\mcb}{\mathcal{B}}
\newcommand{\mcc}{\mathcal{C}}
\newcommand{\mcn}{\mathcal{N}}
\newcommand{\mcd}{\mathcal{D}}

\newtheorem{assumption}{Assumption}
\newtheorem{definition}{Definition}
\newtheorem{remark}{Remark}
\newtheorem{lemma}{Lemma}
\newtheorem{theorem}{Theorem}

\begin{document}

\title{A Discrete Adapted Hierarchical Basis Solver For Radial
Basis Function Interpolation}

\author{Julio E. Castrill\'{o}n-Cand\'{a}s
\and Jun Li  \and Victor Eijkhout}



\begin{abstract}
  In this paper we develop a discrete Hierarchical Basis (HB) to
  efficiently solve the Radial Basis Function (RBF) interpolation
  problem with variable polynomial degree. The HB forms an orthogonal
  set and is adapted to the kernel seed function and the placement of
  the interpolation nodes. Moreover, this basis is orthogonal to a set
  of polynomials up to a given degree defined on the interpolating
  nodes. We are thus able to decouple the RBF interpolation problem
  for any degree of the polynomial interpolation and solve it in two
  steps: (1) The polynomial orthogonal RBF interpolation problem is
  efficiently solved in the transformed HB basis with a GMRES
  iteration and a diagonal (or block SSOR) preconditioner. (2) The
  residual is then projected onto an orthonormal polynomial basis. We
  apply our approach on several test cases to study its effectiveness.
\end{abstract}

\maketitle

\section{Introduction}
\label{RBF:intro}

The computational cost for extracting RBF representations can be
prohibitively expensive for even a moderate amount of interpolation
nodes. For an $N$-point interpolation problem using direct methods it
requires $\mc{O}(N^{2})$ memory and $\mc{O}(N^{3})$ computational
cost. Moreover, since many of the most accurate RBFs have globally
supported and increasing kernels, this problem is often badly
conditioned and difficult to solve with iterative methods. In this
paper we develop a fast, stable and memory efficient algorithm to
solve the RBF interpolation problem based on the construction of a
discrete HB.

Development of RBF interpolation algorithms has been widely studied in
scientific computing. In general, current fast solvers are not yet
optimal.  One crucial observation of the RBF interpolation problem is
that it can be posed as a discrete form of an integral equation. This
insight allows us to extend the techniques originally introduced for
integral equations to the efficient solution of RBF interpolation
problems.

RBF interpolation has been studied for several decades. In 1977 Duchon
\cite{duchon1977} introduced one of the most well known RBFs, the
thin-plate spline. This RBF is popular in practice since it leads to
minimal energy interpolant between the interpolation nodes in 2D.  In
\cite{franke1982} Franke studied the approximation capabilities of a
large class of RBFs and concluded that the biharmonic spline and the
multiquadric give the best approximation.  Furthermore, error
estimates for RBF interpolation have been developed by Schaback et
al. \cite{Wu1993,schaback1995,schaback1999} and more recently by
Narcowich et al.  \cite{narcowich2004}.

RBFs are of much interest in the area of visualization and animation.
They have found applications to point cloud reconstructions, denoising
and repairing of meshes \cite{carr2001}. In general, they have been
used for the reconstruction of 3-D objects and deformation of these
objects \cite{Noh2005}. For these areas of
applications it is usually sufficient to consider zero and
linear-degree polynomials in the RBF problems. However, other applications,
such as Neural Networks and classification \cite{Yee2001}, boundary
and finite element methods \cite{Duan2008,Duan2006}, require
consideration of higher-degree polynomials.

More recently, the connection between RBF interpolation, {\it
  Generalized Least Squares (GLSQ)} \cite{sacks1989} and its extension
to the Best Unbiased Linear Estimator (BLUE) problem has been
established \cite{Nielsen2005,Lophaven2002a,Lophaven2002b}.  If the
covariance matrix of a GLSQ (and BLUE) problem is described by a
symmetric kernel matrix of an RBF problem among other conditions, the
two problems become equivalent.  Although GLSQ is of high interest to
the statistics community, as shown by the high number of citations of
\cite{sacks1989}, the lack of fast solvers limits its application to
small to medium size problems \cite{Lophaven2002a,Lophaven2002b}.
Moreover, many of these statistical problems involve higher than zero-
and linear-degree polynomial regression
\cite{Yalavarthy2007,Yalavarthy2007a,Romero2000,Martin2004,Martin2005,Simpson98}. By
exploiting the connection between GLSQ and RBFs, we will be able to
solve GLSQ using the fast solvers developed in the RBF and integral
equation communities.

For the BLUE Kriging estimator there is less need of higher order
polynomials. In many cases quadratic is sufficient for high accuracy
estimation.  The quadratic interpolant leads to
much better estimate than constant or linear.  In addition, in
\cite{Astrid:2007} the author uses second degree polynomial BLUE for
repairing surfaces.

Recently Gumerov et al. \cite{Gumerov2007} developed a RBF solver with
a Krylov subspace method in conjunction with a preconditioner
constructed from Cardinal functions.  We note that this approach, to
our knowledge, is the state of the art for zero-degree interpolation
in $\mathbb{R}^{3}$ with a biharmonic spline. This makes it very
useful for interpolation problems in computer graphics. On the other
hand, its application to regression problems such as GLSQ is limited.

A domain decomposition method was developed in \cite{beatson2000} by
Beatson et al. This method is a modification of the Von Neumann's
alternating algorithm, where the global solution is approximated by
iterating a series of local RBF interpolation problems. This method is
promising and has led to (coupled with multi-pole expansions) $O(N\log
(N))$ computational cost for certain interpolation problems.

Although the method is very efficient and exhibits $\mathcal{O}(N\log (N))$
computational complexity, this seems to be true for small to medium
size problems (up to 50,000 nodes in $\mathbb{R}^{3}$) with smooth
data. Beyond that range the computational cost increases quadratically
as shown in \cite{beatson2000}. Other results for non smooth data
shows that the computational complexity is more erratic
\cite{carr2003}.  Furthermore, in many cases, it is not obvious how to
pick the optimal domain decomposition scheme.

An alternative approach was developed by Beatson et
al. \cite{beatson1999}, which is based on preconditioning and coupled
with GMRES iterations \cite{saad1986}.  This approach relies on the
construction of a polynomial orthogonal basis, similar to the HB
approach in our paper.  This approach gives rise to a highly sparse
representation of the RBF interpolation matrix that can be very easily
preconditioned by means of a diagonal matrix.  The new system of
equations exhibits condition number growth of no more than
$\mathcal{O}(\log{N})$.  The downside is that this basis is not
complete. This is ameliorated by the introduction of non decaying
elements, but no guarantees on accuracy can be made.

Our approach is based on posing the RBF interpolation problem as a
discretization of an integral equation and applying preconditioning
techniques. This approach has many parallels with the work developed
by Beatson et al. \cite{beatson1999}. However, our approach was
developed from work done for fast integral equation solvers.

Most of the work in the area of fast integral equation solvers has
been restricted to the efficient computation of matrix vector products
as part of an iterative scheme.  For the Poisson kernel the much
celebrated multi-pole spherical harmonic expansions leads to a fast
summation algorithm that reduces each matrix-vector multiplication to
$O(N)$ computational steps \cite{greengard1997,beatson1997}. This
technique has been extended to a class of polyharmonic splines and
multiquadrics \cite{beatson2001,cherrie2002}. More recently L. Ying et
al. has developed multipole algorithms for a general class of kernels
\cite{ying2004}.  In contrast, the development of optimal (or good)
preconditioners for integral equations has been more limited.

A unified approach for solving integral equations efficiently was
introduced in \cite{alpert1993a,alpert1993b,beylkin1991}.  A wavelet
basis was used for sparsifying the discretized operator and only
$O(N\log _2^2 (N))$ entries of the discretization matrix are needed to
achieve optimal asymptotic convergence. The downside is that it was
limited to 1D problems.

In \cite{Castrillon2003} a class of multiwavelets based on a
generalization of Hierarchical Basis (HB) functions was introduced for
sparsifying integral equations on conformal surface meshes in
$\mathbb{R}^{3}$.  These wavelets are continuous, multi-dimensional,
multi-resolution and spatially adaptive.  These constructions are
based on the work on {\it Lifting} by Schroder and Sweldens \cite{ss1}
and lead to a class of adapted HB of arbitrary polynomial degree.  A
similar approach was also developed in \cite{tausch2003}.

These constructions provide compression capabilities that are
independent of the geometry and require only $O(N\log _{4}^{3.5} (N))$
entries to achieve optimal asymptotic convergence. This is also true
for complex geometrical features with sharp edges. Moreover, this
basis has a multi-resolution structure that is related to the BPX
scheme \cite{pasciak1990}, making them an excellent basis to
precondition integral and partial differential equations. In
\cite{heedene2005} Heedene et al.  demonstrate how to use this basis
to build scale decoupled stiffness matrices for partial differential
equations (PDEs) over non uniform irregular conformal meshes.

In this paper, we develop a discrete HB for solving isotropic RBF
interpolation problems efficiently.  Our HB construction is adapted to
the topology of the interpolating nodes and the kernel. This new basis
decouples the polynomial interpolation from the RBF part, leading to a
system of equations that are easier to solve.  With our sparse SSOR
\cite{gene1996,Kolotilina1989} or diagonal preconditioner, combined
with a fast summation method, the RBF interpolation problem can be
solved efficiently.

Our contributions include a method with asymptotic complexity costs
similar to Gumerov et al \cite{Gumerov2007} for problems in
$\mathbb{R}^{3}$.  However, their approach is restricted to only
constant degree RBF interpolation.  Due to the decoupling of the
polynomial interpolation, our approach is more flexible and works well
for higher degree polynomials. We show similar results for the
multiquadric RBFs in $\mathbb{R}^{3}$. In contrast we did not observe multiquadric
results for $\mathbb{R}^{3}$ in \cite{Gumerov2007} and to our
knowledge this result is not available. Note that the idea of decoupling the
RBF system of equations from the polynomial interpolation has also
been proposed in \cite{sibson1991} and \cite{beatson2000}.

In the rest of Section \ref{RBF:intro} we explicitly pose the RBF
interpolation problem. In Section \ref{ADHBC}, we construct an HB that
is adapted to the interpolating nodes and the kernel seed function. In
Section \ref{RBF:Mutli} we demonstrate how the adapted HB is used to
form a multi-resolution RBF matrix, which is used to solve the
interpolation problem efficiently. In section \ref{results}, we show
some numerical results of our method.  The interpolating nodes are
randomly placed, moreover the interpolating values themselves contain
random noise. We summarize our conclusions in section
\ref{conclusions}.

During the writing of this paper we became aware of the H-Matrix
approach by Hackbusch \cite{Borm2003} applied to stochastic
capacitance extraction \cite{Zhu2005} problem. In \cite{Borm2007} the
authors apply an H-matrix approach to sparsify the kernel matrix
arising from a Gaussian process regression problem to $\mc{O}(N
\log{N})$. In our paper, we apply HB to precondition the RBF system,
although we could also use them to sparsify it. Instead, we use a fast
summation approach to compute the matrix-vector products.

\subsection{Radial Basis Function Interpolation}
\label{RBFI}

In this section we pose the problem of RBF interpolation for bounded
functions defined on $\mathbb{R}^{3}$. Although our exposition is only
for $\mathbb{R}^{3}$, the RBF problem and our HB approach can be
extended to any finite dimension.

Consider a function $f(\vec{x}):{\rm \mathbb{R}}^3 \to {\rm
  \mathbb{R}}$ in $L_{\infty}(\mathbb{R}^{3})$ and its evaluation on a
set of user-specified sampling of distinct nodes $X:=\{\vec{x}_1
,...,\vec{x}_N \}\subset {\rm \mathbb{R}}^3$, where $\vec{x} =
[x_{1},x_{2},x_{3}]^{H}$, unisolvent with respect to all polynomials
of degree at most $m$. We are interested in constructing
approximations to $f(\vec{x})$ of the form
\[
s(\vec{x})= \sum_{i=1}^{M(m)} c[i]q_{i}(\vec{x})+\sum\limits_{j=1}^N
{u[j] } K (\vec x,\vec{x}_j ),
\]
\noindent
where $K:\mathbb{R}^{3} \times \mathbb{R}^{3} \rightarrow \mathbb{R}$,
$u \in \mathbb{R}^{N}, c \in \mathbb{R}^{M(m)}$ and
$P:=\{q_{1}(\vec{x}), \dots, q_{M(m)}(\vec{x}) \}$ is a basis for
$\mcp^{m}(\mathbb{R}^{3})$, i.e. the set of all polynomials of total
degree at most $m$ in $\mathbb{R}^{3}$ (Note that $M(m)$ is the number
of polynomials that form a basis for $\mcp^{m}(\mathbb{R}^{3})$ i.e.
$M(m)=\left(m+3 \atop 3\right)$ ).  This interpolant must satisfy the
following condition
\[
s(\vec{x}_{j} )=f(\vec{x}_{j}), \quad j = 1, \dots, N,
\]
\noindent for all $\vec{x}_{j}$ in $X$. Moreover, to ensure the
interpolation is unique we add the following constraint
\begin{equation}
\label{eq2}
\sum\limits_{j=1}^N {u[j] q\left( {\vec{x}_j } \right)=0},
\end{equation}
\noindent for all polynomials $q(\vec{x})$ of degree at most $m$. Now,
since $M(m)$ is the minimum amount of nodes needed to solve the
polynomial problem, we need at least $N \geq M(m)$ RBF centers.  The
interpolation problem can be rewritten in matrix format as
\begin{eqnarray}
\left( {{\begin{array}{*{20}c}
 K \hfill & Q \hfill \\
 {Q^H} \hfill & O \hfill \\
\end{array} }} \right)\left( {{\begin{array}{*{20}c}
 u \hfill \\
 c \hfill \\
\end{array} }} \right)=\left( {{\begin{array}{*{20}c}
 d \hfill \\
 0 \hfill \\
\end{array} }} \right),
\label{RBFI:problem}
\end{eqnarray}
\noindent where $K_{i,j} =K (\vec{x}_i ,\vec{x}_j )$ with $i = 1 \dots
N$ and $j = 1 \dots N$; $d \in \mathbb{R}^{N}$ such that
$d_{j}=f(\vec{x}_{j})$; $c \in \mathbb{R}^{M(m)}$; and $Q_{i,j} =q_j
(\vec{x}_i )$ with $i = 1 \dots N$, \, $j = 1 \dots M(m)$.  Denote the
columns of Q as $[q_{1}, \dots, q_{M(m)}]$.  This is the general form
of the RBF interpolation isotropic problem. The properties of this
approximation mostly depend on the seed function $K(\vec x,\vec y)$.
An example of a well known isotropic kernel in $\mathbb{R}^{3}$ is the
biharmonic spline
\begin{equation}
\label{eq4}
K (\vec x,\vec{x}_j ):= K (| \vec x-\vec{x}_j | )
= \left | {\vec
    x-\vec{x}_j } \right |.
\end{equation}
\noindent This is a popular kernel due to the optimal smoothness of
the interpolant \cite{beatson2000}. This kernel has been successfully
applied in point cloud reconstructions, denoising and repairing of
meshes \cite{carr2001}.  More recently, there has been interest in
extensions to anisotropic kernels \cite{Casciola2006,Casciola2007},
i.e.
\[
K (\vec x,\vec{x}_j ):=K \left( {\left| {T_{j}(\vec x-\vec{x}_j )}
    \right| } \right),
\]
\noindent where $T_{j}$ is a $3 \times 3$ matrix.  The stabilization
method introduced in this paper can be extended to solving efficiently
the RBF problem with spatially varying kernels.  By using the
sparsification properties of the adapted HB a sparse representation of
the spatially varying RBF matrix can be constructed in optimal time.
However, in this paper we restrict the analysis to isotropic kernels
in $\mathbb{R}^{3}$, i.e.  $T_{j}=\alpha I$ where $\alpha > 0$.

One aspect of RBF interpolation is the invertibility of the matrix in
Equation (\ref{RBFI:problem}). In \cite{micchelli1986} it is shown
that the interpolation problem (\ref{RBFI:problem}) has a unique
solution if we assume that the interpolating nodes in $X$ are
unisolvent with respect to $\mcp^{m}(\mathbb{R}^{3})$ and the
continuous kernel is {\it strictly conditionally positive (or
  negative) definite}.  Before we give the definition, we provide some
notation.
\begin{definition} Suppose that $X \subset \mathbb{R}^{3}$ is a set of
  interpolating nodes and $\{q_{1}(\vec{x}),$ $q_{2}(\vec{x}),$
  $\dots, q_{M(m)}(\vec{x})\}$ is a basis for
  $\mcp^{m}(\mathbb{R}^{3})$, then we use $\mcp^{m}(X)$ to denote the
  column space of $Q$.
\end{definition}
\noindent We now assume the kernel matrix $K$ satisfies the following
assumption.
\begin{definition} We say that the symmetric function
$K(\cdot, \cdot) : \mathbb{R}^{N} \times \mathbb{R}^{N} \rightarrow
  \mathbb{R}$ is {\it strictly conditionally positive definite} of
  degree $l$ if for all sets $X \subset \mathbb{R}^{3}$ of distinct
  nodes
\[
v^{H} K v = \sum_{i,j=1}^{N}v_{i}v_{j}K(\vec{x}_{i}, \vec{x}_{j})> 0,
\]
\noindent for all $v \in \mathbb{R}^{N}$ such that $v \perp
\mcp^{l}(X)$ and $v \neq 0$. Alternatively, under the same
assumptions, $K(\cdot, \cdot) : \mathbb{R}^{N} \times \mathbb{R}^{N}
\rightarrow \mathbb{R}$ is {\it strictly conditionally negative
  definite} if
\[
v^{H} K v < 0,
\]
\noindent  for all $v \in \mathbb{R}^{N}$ such that $v \perp
\mcp^{l}(X)$.
\label{errorbounds:asmp1}
\end{definition}
The invertibility of the RBF interpolation problem can be proven by
the basis construction developed in this paper. Although this is not
necessary, it does cast insights on how to construct a basis that can
solve the RBF Problem (\ref{RBFI:problem}) efficiently.
\subsection{Decoupling of the RBF interpolation problem}
\label{Decouple}
Suppose there exists a matrix $T:\mathbb{R}^{N-M} \rightarrow
\mathbb{R}^{N}$, where $M : = dim(\mcp^{m}(X) )$, such that
$T^{H}$ annihilates any vector $v \in \mcp^{m}(X)$ (i.e. $T^{H}v =
0 \,\,\,\,\,\, \forall v \in \mcp^{m}(X)$). Furthermore, suppose
there exists a second matrix $L:\mathbb{R}^{M} \rightarrow
\mathbb{R}^{N}$ such that the combined matrix $P := [L \,\,\,\, T]$ is
orthonormal such that $P^{H}:\mathbb{R}^{N} \rightarrow
\mathbb{R}^{N}$ maps $\mathbb{R}^{N}$ onto
\[
\mcp^{m}(X) \oplus W,
\]
\noindent where $W :=(\mcp^{m}(X))^{\perp}$.  Suppose that $u \in
\mcp^{m}(X)^{\perp}$, then $u = Tw$ for some $w \in
\mathbb{R}^{N-M}$. Problem (\ref{RBFI:problem}) can now be re-written
as
\[
T^{H}KTw + T^{H}Qc = T^{H}d.
\]
\noindent However, since the columns of $Q$ belong in $\mcp^{m}(X)$ then
\begin{equation}
T^{H}KTw = T^{H}d.
\label{RBF:eqn1}
\end{equation}
\noindent From Definition \ref{errorbounds:asmp1} and the
orthonormality of $P$ we conclude that $w$ can be solved uniquely.
The second step is to solve the equation $L^{H}Qc = L^{H}d -
L^{H}KTw$. From the unisolvent property of the nodes X the matrix $Q$
has rank $dim( \mcp^{m}(\mathbb{R}^{3}))$, moreover, $L$ also has
rank $dim( \mcp^{m}(\mathbb{R}^{3}))$, thus $L^{H}Q$ has full rank
and it is invertible.


Although proving the existence of $P$ and hence the uniqueness of the
RBF problem is an interesting exercise, there are more practical
implications to the construction of $P$. First, the coupling of $Q$
and $K$ can lead to a system of ill-conditioned equations depending on
the scale of the domain \cite{beatson2000}. The decoupling property of
the transform $P$ leads to a scale independent problem, thus
correcting this source of ill-conditioning. But more importantly, we
focus on the structure of $T^{H}KT$ and how to exploit it to solve the
RBF interpolation problem (\ref{RBFI:problem}) efficiently.  The key
idea is the ability of $T^{H}$ to vanish discrete polynomial moments
and its effect on the matrix $K(\cdot,\cdot)$. We shall now restrict
our attention to Kernels that satisfy the following assumption.
\begin{assumption} Let $D^\alpha_{x} := \frac{\partial^{\alpha_{1},
\alpha_{2}, \alpha_{3}}}
{\partial \vec{x}_{1}^{\alpha_{1}} \partial
\vec{x}_{2}^{\alpha_{2}} \partial \vec{x}_{3}^{\alpha_{3}}}$ and
similarly for $D^{\beta}_{y}$, we assume that
\[
D^{\alpha}_{x}D^{\beta}_{y} K(\vec x,\vec y) \leq \frac{C}{
  |\vec{x}-\vec{y}|^{q + |\alpha| + |\beta|} },
\]
\noindent where $\alpha= (\alpha_{1},\alpha_{2},\alpha_{3}) \in
\mathbb{Z}{^3}$, $|\alpha| = \alpha_{1}+\alpha_{2}+\alpha_{3}$, and
$q\in \mathbb{Z} $. In addition, we assume that $K(\vec x, \cdot)$ and
$K( \cdot,\vec y)$ are analytic everywhere except for $\vec x = \vec
y$.
\label{errorestimates:assumption2}
\end{assumption}
This assumption is satisfied by many practical kernels, such as
multiquadrics and polyharmonic splines \cite{franke1982,beatson2000}.
\section{Adapted Discrete Hierarchical Basis Constructions}
\label{ADHBC}
In this section we show how to construct a class of discrete HB that
is adapted to the kernel function $K(\cdot,\cdot)$ and to the local
interpolating nodes (or interpolating nodes) contained in $X$. The
objective is to solve RBF interpolation Problem (\ref{RBFI:problem})
efficiently. The HB method will be divided into the following parts:
\begin{itemize}
\item {\bf Multi-resolution domain decomposition.}  The first part is
in essence a preprocessing step to build cubes at different levels
of resolution as place holders for the interpolation nodes belonging
to $X$.
\item {\bf Adapted discrete HB construction.} From the
  multi-resolution domain decomposition of the interpolating nodes in
  $X$, an adapted multi-resolution basis is constructed that
  annihilates any polynomial in $\mcp^{p}(X)$, where $p \in
  \mathbb{Z}^+$ and $p \geq m$. $p$ will be in essence the degree of
  the Hierarchical Basis, which is not to be confused with $m$.
  \item {\bf GMRES iterations with fast summation method.} With the
  adapted HB a multi-resolution RBF interpolation matrix is implicitly
  obtained through a fast summation method and solved iteratively with
  a GMRES algorithm and an SSOR or diagonal preconditioner.
\end{itemize}
\subsection{Multi-resolution Domain Decomposition}
\label{subsection:MRDD}

Without loss of generality, it is assumed that the interpolating nodes
in $X$ are contained in a cube $B^{0}_{0} := [0,1]^{3}$. The next
step is to form a series of level dependent cubes that serve as place
holders for the interpolating nodes at each level of resolution.

The basic algorithm is to subdivide the cube $B^{0}_{0}$ into eight
cubes if $|B^{0}_{0}| > M(p)$, where $|B^{j}_k|$ denotes the total
number of interpolating nodes contained in the cube $B_j^k$.
Subsequently, each cube $B^{j}_{k}$ is sub-divided if $|B^{j}_{k}| >
M(p)$ until there are at most $M(p)$ interpolating nodes at the finest
level.  The algorithm is explained more in detail in the following
pseudo-code:

\begin{algorithm}[H]

  \KwIn{$X:=\{ \vec{x}_{1}, \vec{x}_{2},\dots, \vec{x}_{N}$ \},
    $M(p)$} \KwOut{$B^{j}_{k}\,\,\, \forall k \in \{
    \mck(0),\dots,\mck(n) \}, n$}

\Begin{

pre-processing\;





$j \leftarrow 0$;
$B^{0}_{0} \leftarrow [0,1]^{3}$;
$\mck(0) \leftarrow \{ 0 \}$\;

main\;

\While{$|B^{j}_{k}| > M(p)$ {\bf for any} $k \in \mck(j) $} {

$\mck(j + 1) \leftarrow \emptyset$ \;

\For{$k \leftarrow 0$ \KwTo $|\mck(j)|$} {

  forming $B^{j+1}_{8k},\dots,B^{j+1}_{8k+7}$;
$\mck(j+1) \leftarrow \mck(j+1) \cup_{w=0}^{7} 8k+w$\;


}

$j \leftarrow j + 1$ \;

}

$n \leftarrow j\;$
}

\caption{Multi-resolution Domain Decomposition}
\label{MHBC:algorithm1}
\end{algorithm}
\begin{remark}
$\mck(j)$ is an index set for all the cubes at level $j$.
We use $| \mck(j) |$ to denote the cardinality of $\mck(j)$.
\end{remark}
\begin{remark} Finding the distance between any two boxes can
  be performed in $\mco(N(n+1))$ computational steps by applying
  an octree algorithm. Therefore the Multi-resolution Domain
  Decomposition algorithm can be performed in $\mco(N(n+1))$
  computational steps. This can be easily seen since the maximum
  number of boxes at any level $j$ is bounded by $N$ and there is a
  total of $n+1$ levels.
\label{MHBC:remark2}
\end{remark}
Before describing the construction of the adapted discrete HB, we
introduce some more notations to facilitate our discussion.
\begin{definition}
  Let $\mcb_{j}$ be the set of all the cubes $B^{j}_{k}$ at level
  $j$ that contain at least one interpolating center from $X$.
\end{definition}
\begin{definition} Let ${\bf C} := \{ e_{1},\dots,e_{N} \},$ where
  $e_{i}[i]=1$ and $e_{i}[j]=0$ if $i \ne j$.  Furthermore, define the
  bijective mapping $F_{p}: {\bf C} \rightarrow X$ such that $F_{p}
  (e_{i}) = \vec{x}_{i}$, for $i = 1 \dots N$ and $F_{q}: {\bf C}
  \rightarrow Z^{+}$ s.t. $F_{q}(e_{i}) = i$.  Now, for each cube
  $B^{n}_{k} \in \mcb_n$ at the finest level $n$, let
  \[ {\bf B}^{n}_{k} := \{ e_{i}\,\,\,|\,\,\ F_{p} (e_{i}) \in
  B^{n}_{k}\}.
  \]
  \noindent and for all $l = 1, \dots, n-1$

  \[ {\bf \tilde{B}}^{l}_{k} := \{ e_{i}\,\,\,|\,\,\ F_{p} (e_{i}) \in
  B^{l}_{k}\}.
  \]

\end{definition}
\begin{definition}
  Let $\mcc_{n} := \bigcup_{k \in \mck(n)} {\bf B}^{n}_{k}$.
\end{definition}
\begin{definition} For all $j = 0, \dots, n-1$, let
  $children(B^{j}_{k})$ be the collection of nonempty subdivided cubes
  $B^{j+1}_{l} \in \mcb_{j+1}$, for some $l \in \mathbb{N}$, of
  the cube $B^{j}_{k}$.
\end{definition}

\begin{definition} For every non empty $B^{j}_{k}$ let the set
  $parent(B^{j}_{k}) := \{B^{j-1}_{l} \in \mcb_{j-1} \,\,\,|\,\,\,
  B^{j}_{k} \in children(B^{j-1}_{l}) \}$.
\end{definition}
\subsection{Basis Construction}
\label{sub:basisconstruction}

From the output of the multi-resolution decomposition Algorithm
\ref{MHBC:algorithm1} we can now build an adapted discrete HB that
annihilates any polynomial in $\mcp^{p}(X)$. To construct such a
basis, we apply the {\it stable completion} \cite{carnicer1996}
procedure. This approach was followed in
\cite{Castrillon2003}. However, the basis is further orthogonalized by
using a modified Singular Value Decomposition (SVD) orthonormalization
approach introduced in \cite{tausch2003}.

Suppose $v_{1},\dots,v_{s}$ are a set of orthonormal vectors in
$\mathbb{R}^{N}$, where $s \in \mathbb{Z}^{+}$, a new basis is
constructed such that
\[
\phi_{j} := \sum_{i=1}^{s} c_{i,j} v_{i}, \hspace{2mm} j=1, \dots, a;
\hspace{5mm} \psi_{j} := \sum_{i=1}^{s} d_{i,j} v_{i}, \hspace{2mm}
j=a+1, \dots, s,
\]
\noindent where $c_{i,j}$, $d_{i,j} \in \mathbb{R}$ and for some $a
\in \mathbb{Z}^{+}$. We desire that the new discrete HB vector
$\psi_{j}$ to be orthogonal to $\mcp^{p}(X)$, i.e.
\begin{equation}
\sum_{k=1}^{N} r[k] \psi_{j}[k] = 0,
\label{hbconstruction:eqn1}
\end{equation}
\noindent for all $r \in \mcp^{p}(X)$. Notice that the summation
and the vectors $r$ and $\psi_{j}$ are in the same order as the entries
of the set X.

Due to the orthonormality of the basis $\{v_{i}\}_{i=1}^{s}$ this
implies that Equation (\ref{hbconstruction:eqn1}) is satisfied if the
vector $[ d_{i,1} , \dots, d_{i,s} ]$ belongs to the null space of
the matrix
\[
M_{s,p} := Q^{H}V,
\]
\noindent where the columns of $Q$ are a basis for $\mcp^{p}(X)$
(i.e. all the polynomial moments) and $V = [ v_{1}, v_{2}, \dots,
v_{s} ]$.  (Notice that the order of the summation is done with
respect to the set X). Suppose that the matrix $M_{s,p}$ is a rank $a$
matrix and let $U_{s,p}D_{s,p}V_{s,p}$ be the SVD decomposition. We
then pick
\begin{equation}
  \left[ \begin{array}{lll|lll}
      c_{0,1} & \dots &c_{a,1} & d_{a+1,1} & \dots &d_{s,1} \\
      c_{0,2} & \dots &c_{a,2} & d_{a+1,2} & \dots &d_{s,2} \\
      \vdots & \vdots & \vdots & \vdots & \vdots & \vdots   \\
      c_{0,s} & \dots &c_{a,s} & d_{a+1,s} & \dots &d_{s,s}
    \end{array}
\right] := V^{H}_{s,p},
\label{eqDefVspT}
\end{equation}
\noindent where the columns $a+1$, \dots, $s$ form an orthonormal
basis of the nullspace $\mcn(M_{s,p})$. Similarly, the columns
$1,\dots, a$ form an orthonormal basis of $\mathbb{R}^{N} \backslash
\mcn(M_{s,p})$.
\begin{remark} If $\{v_{1},\dots,v_{s}\}$ is orthonormal, then new
  basis $\{\phi_{1},\dots,\phi_{a}, $ $\psi_{a+1},\dots,$ $\psi_{s}
  \}$ is orthonormal, and spans the same space as $span
  \{v_{1},\dots,v_{s}\}$.  This is due to the orthonormality of the
  matrix $V_{s,p}$.
\end{remark}
\begin{remark} If $s$ is larger than the total number
of vanishing moments, then $M_{s,p}$ is guaranteed to have a nullspace
of at least rank $s - M(p)$, i.e. there exist at least $s - M(p) $
orthonormal vectors $\{ \psi_{i} \}$ that satisfy Equation
\ref{hbconstruction:eqn1}.
\end{remark}
\subsection{Finest Level}
We can now build an orthonormal multi-resolution basis.  First, choose
a priori the degree of moments $p$ and start at the finest level
$n$. The next step is to progressively build the adapted HB as the
levels are traversed.

At the finest level $n$, for each cube ${\bf B}^{n}_{k} \in \mcc_{n}$
let $v_{i} := e_{i}$ for all $e_{i} \in {\bf B}^{n}_{k}$.  As
described in the previous section, the objective is to build new
functions
\[
\phi_{k,l}^{n} := \sum_{i=1}^{s} c_{n,i,l,k} v_{i}, \hspace{2mm} l=1,
\dots, a_{n,k},
\hspace{5mm} \psi^{n}_{k,l} := \sum_{i=1}^{s} d_{n,i,l,k} v_{i},
\hspace{2mm} l=a_{n,k}+1, \dots, s,
\]
\noindent such that Equation (\ref{hbconstruction:eqn1}) is satisfied.

The first step is to form the matrix $M^{n,k}_{s,p}:=Q^{H}V$, where
the columns of $Q$ are a basis for $\mcp^{p}(X)$. Notice that
since $e_{i}[w] = 0 $ for $w \neq i$ and $e_{i} \in {\bf B}^{n}_{k} $,
then only $|B^{n}_{k}|$ columns of $Q^{H}$ are needed to form the
matrix $M^{n,k}_{s,p}$ and the rest can be thrown away since they
multiply with zero.

The next step is to apply the SVD procedure such that $M^{n,k}_{s,p}
\rightarrow U^{n,k}_{s,p}D^{n,k}_{s,p}V^{n,k}_{s,p}$. The coefficients
$c_{n,i,j,k}$ and $d_{n,i,j,k}$ are then obtained from the rows of
$V^{n,k}_{s,p}$ and $a_{n,k} := {\it rank} M_{s,p}$.

Now, for each ${ B}^{n}_{k} \in \mcb^{n}$ denote $\bar{C}^{n}_{k}$
as the collection of basis vectors $\{ \phi^{n}_{k,1}, \dots,$
$\phi^{n}_{k, a_{n,k}} \}$, and similarly denote $\bar{D}^{n}_{k}$ as
the collection of basis vectors $\{ \psi^{n}_{k,a_{n,k}+1},\dots, $ $
\psi^{n}_{k,s}\}$.  Furthermore, we define the {\it detail} subspace
\[
W^{n}_{k} := span \{ \psi^{n}_{k,a_{n,k}+1}, \dots, \psi^{n}_{k,s
} \}
\]
\noindent and the {\it average} subspace
\[
V^{n}_{k} := span \{ \phi^{n}_{k,1}, \dots, \phi^{n}_{k, a_{n,k}} \}.
\]
\noindent By collecting the transformed
vectors from all the cubes in $\mcc_{n}$, we form the subspaces
\[
V^{n} := \oplus_{k \in \mck(n)} V^{n,k}, \,\,\,
W^{n} := \oplus_{k \in \mck(n)} W^{n,k},
\]
\noindent where $\oplus$ is a direct sum and $\mck(i) := \{k
\,\,\,|\,\,\, B^{i}_{k} \in \mcb^{i} \}$.
\begin{remark}
  We first observe that $\mathbb{R}^{N} = V^{n} \oplus W^{n} \oplus
  \tilde{V}_{n}$, where $\tilde{V}_{n}$ is the span of all the unit
  vectors contained in $\{ \{ \tilde{\bf B}^{l}_{k} \}_{k \in \mck(l)}
  \}_{l=1}^{n-1}$.  This is true since the number of interpolating
  nodes is equal to $N$ and $\mathbb{R}^{N} = span \{e_{1}, e_{2},
  \dots, e_{N} \}$.
\label{hbconstruction:remark5}
\end{remark}
\begin{remark}
  It is possible that $W^{n}_{k} = \emptyset$ for some particular cube
  ${B}^{n}_{k}$. This will be the case if the cardinality of ${\bf
    B}^{n}_{k}$ is less or equal to $M(p)$ i.e. the dimension of the
  nullspace of $M_{s,p}$ is zero. However, this will not be a problem.
  As we shall see in section \ref{section:intermediate}, the next set
  of HB are built from the vectors in $\bar{C}^{n}_{k}$ and its
  siblings.
\end{remark}
\begin{lemma}
The basis vectors of $V^{n}$ and $W^{n}$ form an orthonormal set.
\label{hbconstruction:lemma7}
\end{lemma}
\begin{proof}
  First notice that since ${\bf B}^{n}_{l} \cap {\bf B}^{n}_{k} =
  \emptyset $ whenever $k \neq l$ then $V^{n,k} \perp V^{n,l}$,
  $W^{n,k} \perp W^{n,l}$ and $V^{n,k} \perp W^{n,l}$. The result
  follows from the fact that the rows $V^{n,k}_{s,p}$ form an
  orthonormal set.
\end{proof}
It is clear that the {\it detail} subspace $W^{n} \perp \mcp^{p}(\vec{x})$, but the {\it average} subspace $V^{n}$ is
not. However, we can still perform the SVD procedure to further
decompose $V^{n}$. To this end we need to accumulate the {\it average}
basis vectors of $V^{n}$ and all the unit basis vectors in $\{{\bf
  \tilde{B}}^{n-1}_{k \in \mck(n-1)} \}$.  For each $B^{n-1}_{k}
\in \mcb_{n-1}$ identify the set $children(B^{n-1}_{k})$. Form the
set ${\bf B}^{n-1}_{k} :=\{\bar{C}^{n}_{l} \,\,\,|\,\,\, B^{n}_{l} \in
children(B^{n-1}_{k}) \} $.  If $B^{n_1}_{k}$ has no children then
${\bf B}^{n-1}_{k} = \tilde{\bf B}^{n-1}_{k}$.  We can now apply the
SVD procedure on each set of average vectors in ${\bf B}^{n-1}_{k}$.
\subsection{Intermediate Level}
\label{section:intermediate}
Suppose we have the collection of sets ${\bf B}^{i}_{k}$ for all $k
\in \mck(i)$. For each ${\bf B}^{i}_{k}$ perform the matrix
decomposition $M^{i,k}_{s,p} =
U^{i,k}_{s,p}D^{i,k}_{s,p}V^{i,k}_{s,p}$ for all $v \in {\bf
  B}^{i}_{k}$.  From the matrix $V^{i,k}_{s,p}$ obtain the
decomposition
\[
\phi_{k,l}^{i} := \sum_{j=1}^{s} c_{i,j,l,k} v_{j}, \hspace{2mm} l=1,
\dots, a_{i,k}, \hspace{5mm} \psi^{i}_{k,l} := \sum_{j=1}^{s}
d_{i,j,l,k} v_{j}, \hspace{2mm} l=a_{i,k}+1, \dots, s,
\]
\noindent where the coefficients $c_{i,j,l,k}$ and $d_{i,j,l,k}$ are
obtained from the rows of $V^{i,k}_{s,p}$ in (\ref{eqDefVspT}) and
$a_{i,k} := {\it rank} M^{i,k}_{s,p}$.  Then we form the subspaces
\[
W^{i}_{k} := span \{ \psi^{i}_{k,a_{i,k}+1}, \dots, \psi^{i}_{k, s}
\}, \,\,\, V^{i}_{k} := span \{ \phi^{i}_{k,1}, \dots, \phi^{i}_{k,
  a_{i,k}} \},
\]
\noindent and
\[
V^{i} := \oplus_{k \in \mck(i)} V^{i,k}, \,\,\, W^{i} := \oplus_{k
  \in \mck(i)} W^{i,k}.
\]
\noindent It is easy to see that $V^{i+1} = V^{i} \oplus W^{i} \oplus
\tilde{V}_{i}$, , where $\tilde{V}_{i}$ is the span of all the unit
vectors contained in $\{ \{ \tilde{\bf B}^{l}_{k} \}_{k \in \mck(l)}
\}_{l=1}^{i}$. The basis vectors are collected into two groups:
\begin{definition}
  For each ${B}^{i}_{k} \in \mcb^{i}$ that have children let the
  sets, for $i = 0,\dots,n-1$, $\bar{C}^{i}_{k} := \{ \phi^{i}_{k,1},
  \dots, \phi^{i}_{k, a_{i,k}} \},$ and $\bar{D}^{i}_{k} :=$ $\{$
  $\psi^{i}_{k,a_{n,k}+1},\dots, $ $ \psi^{i}_{k,s }\}$.
\end{definition}
Just as for the finest level case, we can further decompose
$V^{i}$. To this end, for each $B^{i-1}_{k} \in \mcb_{i-1}$
identify the set $children(B^{i-1}_{k})$ and form the set ${\bf
  B}^{i-1}_{k}$ $:=\{\bar{C}^{i}_{l}$ $ \,\,\,|\,\,\, B^{i}_{l} \in$
$children(B^{i-1}_{k}) \}$.  If $B^{i-1}_{k}$ has no children then
${\bf B}^{i-1}_{k} = \tilde{\bf B}^{i-1}_{k}$.
\subsection{Coarse Level}
It is clear that when the iteration reaches $V^{0}$ the basis function
no longer annihilates polynomials of degree $p$. However, a new basis
can be obtained that can vanish polynomials of degree $m$.

Recall that for the RBF interpolation problem with polynomial degree
$m$ it is imposed that $u \perp \mcp^{m}(X)$. If $p = m$ then it
is clear that $u \in W^{0} \oplus \dots W^{n}$ and RBF problem
decouples as shown in Section \ref{RBF:intro}. However, if $p > m$
then $u \in (\mcp^{p}(X) \backslash \mcp^{m}(X)) \oplus W^{0}
\oplus \dots W^{n}$ and the RBF problem does not decouple. It is then
of interest to find an orthonormal basis to $\mcp^{p}(X)
\backslash \mcp^{m}(X).$ This can be easily achieved.  Let the
columns of the matrix $Q$
be a basis for $\mcp^{p}(X)$, where each function $q_{i}(x)$
corresponds to the $i^{th}$ moment. Now, the first $M(m)$ columns
correspond to a basis for $\mcp^{m}(X)$. Thus an orthonormal basis
for $\mcp^{p}(X) \backslash \mcp^{m}(X)$ is easily achieved by
applying the Gram-Schmidt process.

Alternatively the matrix $M_{0,m}$ can now be formed by applying the
SVD decomposition and a basis that annihilates all polynomial of
degree $m$ or lower is obtained. The matrix $\bar{C}_{0}^{0}$ can now
be replaced with the matrix $[C^{-1}_{0}, D^{-1}_{0} ]$, where the
columns of $C^{-1}_{0}$ form an orthonormal basis for $\mcp^{m}(X)$
and $D^{-1}_{0}$ is an orthonormal basis for $\mcp^{p}(X) \backslash
\mcp^{m}(X)$.

The complete algorithm to decompose $\mathbb{R}^{N}$ into a
multi-resolution basis with respect to the interpolating nodes $X$ is
described in Algorithms \ref{MHBC:algorithm2} and
\ref{MHBC:algorithm3}.

\begin{algorithm}[H]

      \KwIn { Finest level $n$; Degree of RBF $m$; $B_{k}^{j} \,\,\,
        \forall k \in \mck(j), j = -1\dots n$; ${\bf B}^{n}_{k}
        \forall k \in \mck(n)$;  $\{ \{ \tilde{\bf B}^{l}_{k} \}_{k \in \mck(l)} \}_{l=1}^{n}$;
        Degree of vanishing moments $p \geq m$;
         $X$.}

      \KwOut{$\{ \bar{C}^{-1}_{0}, \bar{D}^{-1}_{0}, \bar{D}^{0}_{0},
        \bar{D}^{0}_{1}, \dots, \bar{D}^{n}_{k} \}$}

main\;

\For{$j \leftarrow n$ \KwTo $1$ {\bf step} $-1$}{


\For{$k \leftarrow 1$ \KwTo $|\mck(j-1)|$}{
    {\bf B}$_k^{j-1} \leftarrow \emptyset$
}
\For{$k \leftarrow 1$ \KwTo $|\mck(j)|$}{

$\{ \bar{D}^{j}_{k}$, $\bar{C}^{j}_{k} \} \leftarrow $ PolyOrtho({\bf
    B}$^{j}_{k}$, $p$); $U \leftarrow parent(B^{j}_{k})$ \;

\ForAll{$B^{j-1}_{l} \in U$}{

       ${\bf B}^{j-1}_{l} \leftarrow {\bf B}^{j-1}_{l} \cup \bar{C}^{j}_{k}$;

}

$\mbox{\bf forall}\,\,\, {\bf \tilde{B}}^{j-1}_{k} \in \mcb_{j-1}$\, \mbox{\bf let} \,${\bf B}^{j-1}_{k} = {\bf \tilde{B}}^{j-1}_{k}$;

}

}

$\{ \bar{D}^{-1}_{0}$, $\bar{C}^{-1}_{0} \} \leftarrow $ PolyOrtho({\bf
    B}$^{0}_{0}$,$m$)\;

  \caption{Adapted Discrete HB Construction}
\label{MHBC:algorithm2}
\end{algorithm}
\begin{lemma}
\[
\mathbb{R}^{N} = V^{0} \oplus W^{0} \oplus \dots W^{n} = span \{
\bar{C}^{-1}_{0}, \bar{D}^{-1}_{0}, \bar{D}^{0}_{0}, \bar{D}^{0}_{1},
\dots, \bar{D}^{n}_{k} \}.
\]
\noindent for $j = 0 \dots n$ and for all $k \in \mck(j)$
\label{properties:lemma9}
\end{lemma}
\begin{proof} The result follows from Remark
  \ref{hbconstruction:remark5} and that $V_{i}$ is decomposed into
  $V^{i-1} \oplus W^{i-1} \oplus \tilde{V}_{i-1}$ for all $i = 1 \dots
  n$.
\end{proof}
\begin{remark} When Algorithm \ref{MHBC:algorithm2}
terminates at level $i = 0$, there will be $M(p)$ orthonormal vectors
that span $\mcp^{p}(X)$.
\end{remark}
\begin{remark} At the finest level $n$, the number of vectors in each
  matrix $\bar{C}^{n}_{k}$ corresponding to ${\bf B}^{n}_{k}$ is
  bounded by $M(p)$. Now, for each ${\bf B}^{n-1}_{k}$ there are at
  most $8M(p)$ vectors from the children of ${\bf B}^{n-1}_{k}$. From
  the procedure for the basis construction in section
  \ref{sub:basisconstruction} for each ${\bf B}^{n-1}_{k}$ there are
  at at most $M(p)$ vectors in $\bar{C}^{n}_{k}$.  Furthermore, there
  are no more than $8M(p)$ vectors in $\bar{D}^{n}_{k}$ formed.  The
  same conclusion follows for each ${\bf B}^{i}_{k}$, for all levels
  $i = 0,\dots,n$.
\label{hbconstruction:remark2}
 \end{remark}

\begin{algorithm}[H]

      \KwIn{${\bf B}^{j}_{k}$, Degree of vanishing moment $p$}

      \KwOut{$\bar{D}^{j}_{k}$, $\bar{C}^{j}_{k}$ }


$\bar{C}^{j}_{k} \leftarrow \emptyset$;\,\,\,, $\bar{D}^{j}_{k} \leftarrow \emptyset$
$s \leftarrow |{\bf B}^{j}_{k}|$;\,\,
$V \leftarrow [v_{1}, \dots, v_{s}]$; \,\,
        $M^{j,k}_{s,p} \leftarrow Q^{H}V$\;
        $[U^{j,k}_{s,p}, D^{j,k}_{s,p},V^{j,k}_{s,p}] \leftarrow \,\,
SVD(M^{j,k}_{s,p})$; \,\, $a_{j,k} \leftarrow $ rank of $D^{j,k}_{s,p}$;

\For{$l \leftarrow 1$ \KwTo $a_{j,k}$}{


       $\phi_{k,l}^{j} \leftarrow  \sum_{i=1}^{s} c_{j,i,l,k} v_{i}$; \,\,\,
$\bar{C}^{j}_{k} \leftarrow
  [\bar{C}^{j}_{k},\,\,\, \phi^{j}_{k,l}]$\;

}

\For{$l \leftarrow a_{j,k}+1$ \KwTo $s$ }{
       $\psi^{j}_{k,l} \leftarrow  \sum_{i=1}^{s} d_{j,i,l,k} v_{i}$; \,\,\,
$\bar{D}^{j}_{k} \leftarrow [\bar{D}^{j}_{k}, \psi^{j}_{k,l}]$\;
}

\caption{PolyOrtho(${\bf B}^{j}_{k}$,$p$)}
\label{MHBC:algorithm3}
\end{algorithm}
\begin{definition} For any ${\bf B}^{i}_{k}$, $k \in \mck(i)$, i =
  0, \dots n, let $|{\bf B}^{i}_{k}|$ be the number of vectors in
  ${\bf B}^{i}_{k}$.
\end{definition}
\begin{theorem}
  The complexity cost for Algorithm \ref{MHBC:algorithm2} is bounded
  by $\mco(Nn)$.
\label{complexity:theorem1}
\end{theorem}
\begin{proof}
  Suppose we start at the finest level $n$. Now, for each box in
  $B^{n}_{k}$, the vectors $e_{i} \in {\bf B}^{n}_{k}$ have at most
  one non-zero entry. This implies that the matrix
  $M^{n,k}_{s,p}=Q^{H}V$, $Q^{H}$ is a $M(p) \times |B^{n}_{k}|$
  matrix and $V$ is at most a $|B^{n}_{k}| \times |B^{n}_{k}|$
  matrix. Then the total cost to computing $M^{n,k}_{s,p}$ for all $k
  \in \mck(n)$ is bounded by
 \[
 C\sum_{k \in \mck(n)} |B^{n}_{k}|^{2}M(p)
 \]
 for some $C > 0$.  Now since $\left| \cup_{k \in \mck(n)}
   B^{n}_{k} \right| = N$ and $|B^{n}_{k}|$ is at most $M(p)$ $\forall
 k \in \mck(n)$, then the cost for computing $\bar{C}^{n}_{k}$ and
 $\bar{D}^{n}_{k}$, $\forall k \in \mck(n)$, is at most $\mco(N)$.

 At level $n-1$, from Remark \ref{hbconstruction:remark2} we see that
 there are at most $8M(p)$ vectors in each ${\bf B}^{n-1}_{k}$
 $\forall k \in \mck(n-1)$. Forming the the matrix
 $M^{n-1,k}_{s,p}=Q^{H}V$, $Q^{H}$ is at most $M(p) \times
 |B^{n-1}_{k}|$ and $V$ is at most $|B^{n-1}_{k}| \times |{\bf
   B}^{n-1}_{k}|$.  Now, since $\left| \cup_{k \in \mck(n-1)}
 B^{n-1}_{k} \right| = N$ it follows that the cost for computing
 $M^{n-1,k}_{s,p}$, $\forall k \in \mck(n-1)$ , is at most $\mco(N)$.
 Furthermore, we have from Remark \ref{hbconstruction:remark2} that
 $|{\bf B}^{n-2}_{k}| \leq 8M(p)$, $\forall k \in \mck(n-2)$.

 Since for each level $i$, $\left| \cup_{k \in \mck(i)} B^{i}_{k}
 \right| \leq N$, then the total cost of computing $M^{i,k}_{s,p}$,
 $\forall k \in \mck(i)$, is at most $\mco(N)$ and $|{\bf
   B}^{i-1}_{k}| \leq 8M(p)$, $\forall k \in \mck(i-1)$. The
 result follows.
\end{proof}
\subsection{Properties}


The adapted HB construction has some interesting properties. In
particular, the space $\mathbb{R}^{N}$ can be decomposed in a series
of nested subspaces that are orthogonal to $\mcp^{p}(X)$ and the
basis forms an orthonormal set. As a side benefit, this series of
nested subspaces can be used to prove the uniqueness of the RBF
interpolation problem.  One important property of the adapted HB is
presented in the following lemma.
\begin{lemma} The basis of $R^{N}$ described by the vectors of $\{
  \bar{C}^{-1}_{0}, \bar{D}^{-1}_{0},$ $\bar{D}^{0}_{0},$
  $\bar{D}^{0}_{1}, \dots, \bar{D}^{n}_{k} \}$, $j = 0 \dots n$, $k
  \in \mck(j)$ form an orthonormal set.
\label{properties:lemma10}
\end{lemma}
\begin{proof} We prove this by a simple induction argument. Assume
  that for level $i$ the set of vectors $\{ {\bf B}^{i}_{k} \}$ are
  orthonormal. Since the rows of the set $V^{i,k}_{s,p}$ are
  orthonormal and ${\bf B}^{i}_{l} \cap {\bf B}^{i}_{k} = \emptyset $
  whenever $l \neq k$, then it follows that the vectors $ \cup_{k \in
    \mck(i)} \{\bar{C}^{i}_{k}, \bar{D}^{i}_{k} \}$ form an
  orthonormal basis.  The result then follows from Lemma
  \ref{hbconstruction:lemma7}.
\end{proof}
\begin{definition} Given a set of unisolvent interpolating nodes $X
  \subset \mathbb{R}^{3}$ with respect to $\mcp^{p}(\mathbb{R}^{3})$,
  we form the matrix $P$ from the basis vectors
  $\{\bar{C}^{-1}_{0},\bar{D}^{-1}_{0},$ $ \bar{D}^{0}_{0},
  \bar{D}^{0}_{1}, \dots,$ $\bar{D}^{n}_{k} \} $.
\end{definition}
From Lemmas \ref{properties:lemma9} and \ref{properties:lemma10} the
matrix $P$ has the following properties
\begin{enumerate}
\item If $v \in \mcp^{p}(X)$ then $P^{H}v$ has $dim(
\bar{C}^{0}_{0} )$ non-zero entries.
\item $PP^{H} = P^{H}P = I$.
\end{enumerate}
\section{Multi-Resolution RBF Representation}
\label{RBF:Mutli}
The HB we constructed above is adapted to the kernel and the location
of the interpolation nodes.  It also satisfies the vanishing moment
property.  The construction of such an HB leads to several important
consequences.  First, we can use the basis to prove the existence of a
unique solution of the RBF problem, but more importantly, this basis
can be used to solve the RBF problem efficiently.

As the reader might recall from section \ref{RBFI}, the construction
of the adapted HB decouples the polynomial interpolant from the RBF
functions if the degree of the vanishing moments $p$ is equal to the
degree of the RBF polynomial interpolant $m$. This simple result can
be extended if $p \ge m$.
\begin{theorem} Suppose $X$ is unisolvent with respect to $\mathbb{R}^{3}$ and $u$
  solves the interpolation problem of equation (\ref{RBFI:problem}) uniquely, where
  $u \perp \mcp^{m}(X)$ and the kernel satisfies Definition
  \ref{errorbounds:asmp1}. If the number of vanishing moments $p \geq
  m$ then
\begin{eqnarray}
  \left( {{\begin{array}{*{20}c}
          C_{\perp}^{H}KC_{\perp} \hfill & C_{\perp}^{H}KT \hfill \\
          {T^{H} K C_{\perp} } \hfill & T^{H}KT \hfill \\
        \end{array} }} \right)\left( {{\begin{array}{*{20}c}
          s \hfill \\
          w \hfill \\
        \end{array} }} \right)=\left( {{\begin{array}{*{20}c}
          C_{\perp}^{H} d \hfill \\
          T^{H} d \hfill \\
        \end{array} }} \right),
  \label{SRBF:eqn1}
\end{eqnarray}
\noindent for some $s \in \mathbb{R}^{M - O}$ and $w \in R^{N-M}$,
where $T: = [\bar{D}^{0}_{0}, \bar{D}^{0}_{1}, \dots,
  \bar{D}^{n}_{k}]$, $C_{\perp} = \bar{D}^{-1}_{0}$ and $O =
dim(\mcp^{m}(X))$. Moreover, $u = C_{\perp}s + Tw$.
\end{theorem}
\begin{proof}
  Since $u \in (\mcp^{p}(X) \backslash \mcp^{m}(X)) \oplus W^{0}
  \oplus \dots W^{n}$, then $u = C_{\perp}s + Tw$ for some $s \in
  \mathbb{R}^{M - O}$ and $w \in R^{N-M}$, where $O =
  dim(\mcp^{m}(X))$. Replacing $u$ into (\ref{RBFI:problem}),
  pre-multiplying by $[C_{\perp}^{H} T^{H}]^{H}$ and recalling that
  $C_{\perp}, T \perp \mcp^{m}(X)$ the result follows.
\end{proof}

Once $u$ is found, $c$ is easily obtained by solving the set of
equations $L^{H}Qc = L^{H}(d - Ku)$, where $L^{H}Q \in R^{ M(p) \times
  M(p)}$.

There are two ways we can solve this, since $L$ and $Q$ span the same
space and have full column rank, then $L^{H}Q$ is invertible and
\[
c = (L^{H}Q)^{-1} L^{H} (d - Ku).
\]
\noindent Alternatively, we can define the interpolation problem
in terms of the basis vectors in $L$ directly i.e. $Q:=L$, which
leads to
\[
c = L^{H}(d - Ku).
\]
For the rest of this section we describe the algorithms for solving
the previous system of equations. The entries of the matrix
\[
K_{W} := \left( {{\begin{array}{*{20}c}
        C_{\perp}^{H}KC_{\perp} \hfill & C_{\perp}^{H}KT \hfill \\
        {T^{H} K C_{\perp} } \hfill & T^{H}KT \hfill \\
\end{array} }} \right)
\]
\noindent are formed from all the pairwise matching of any two vectors
$\psi^{i}_{k,m},\psi^{j}_{l,g}$ from the set
$\mcd:=\{\bar{D}^{-1}_{0}, \bar{D}^{0}_{0}, \bar{D}^{0}_{1}, \dots,$
$\bar{D}^{n}_{k} \} $. The entries of $K_{W}$ take the form
\begin{equation}
 \sum_{k \in \mck(n)}\sum_{k' \in \mck(n)} \sum_{{e}_{a}
 \in {\bf B}^{n}_{k}} \sum_{e_{b} \in {\bf B}^{n}_{k'}}
 K(F_{p}(e_{a}),F_{p}(e_{b}))
 \psi^{i}_{k,m}[F_{q}(e_{a})]\psi^{j}_{l,g}[F_{q}(e_{b})],
 \label{SRBF:eqn2}
\end{equation}
\noindent Notice that the summation is over all the vectors $e_o$
s.t. $o = 1,\dots, N$. However, the entries of $\psi^{i}_{k,m}$ are
mostly zeros, thus in practice the summation is over all the non-zero
terms.

Continuing with the same notation, the entries of $d_{W} := Td$ have
the form
\[
\sum_{k \in \mck(n)} \sum_{e_{a} \in {\bf B}^{n}_{k}}
\psi^{i}_{k,m}[F_{q}({e}_{a})] f(F_{p}({e}_{a})).
\]
\noindent Since $w = Pu$ and $u \perp \mcp^{p}(X)$, then
entries of $w$ have the form
 \[
 \sum_{k \in \mck(n)} \sum_{e_{a} \in {\bf B}^{n}_{k}}
 \psi^{i}_{k,m}[F_{q}(e_{a})] u[F_{q}(e_{a})],\,\,\, \forall
 \psi^{i}_{k,m} \in \mcd.
 \]
 \noindent It is clear that from the set $\mcd$ the matrix $K_{W}$
 is ordered such that the entries of any row of $K_{W}$ sums over the
 same vector $\psi^{i}_{k} \in \mcd$.
In Figure
\ref{SRBF:fig1} a block decomposition of the matrix $K_{W}$ is shown.
\begin{figure}[htbp]
  \psfrag{a}[l]{$K^{n,n}_{W}$} \psfrag{b}[c]{$\dots$}
  \psfrag{c}[c]{$K^{n,0}_{W}$} \psfrag{d}[c]{$\vdots$}
  \psfrag{e}[l]{$\ddots$} \psfrag{f}[c]{$\vdots$}
  \psfrag{g}[c]{$K^{0,n}_{W}$} \psfrag{h}[c]{$\dots$}
  \psfrag{i}[c]{$K^{0,0}_{W}$} \psfrag{j}[l]{$w^{n}$}
  \psfrag{k}[l]{$\vdots$} \psfrag{m}[l]{$w^{0}$}
  \psfrag{n}[c]{$d_{W}^{n}$} \psfrag{o}[l]{$\vdots$}
  \psfrag{p}[c]{$d_{W}^{0}$}
\begin{center}
\includegraphics[width=4.2in,height=2.10in]{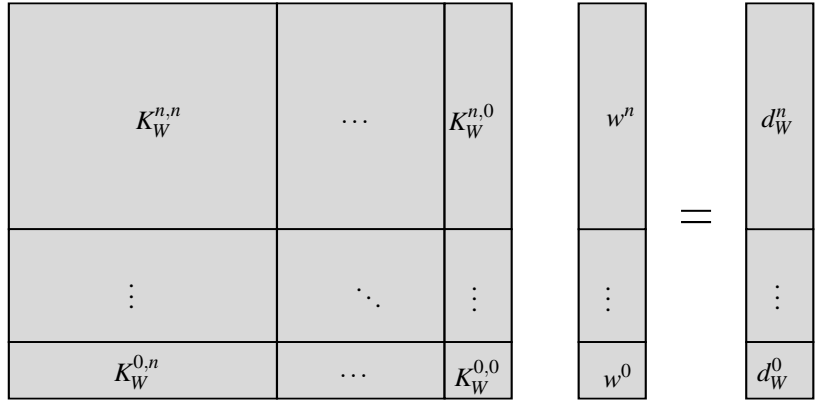}
\caption{Organization of the linear system $K_{W}w = d$. The block
  matrices $K^{ij}_{W}$ consist of all the summations in Equation
  \ref{SRBF:eqn2}, for all $\psi^{i}_{k,m}, \psi^{j}_{l,g} \in \mcd$
  that belong to level $i$ and $j$. The vectors $d^{i}_{W}$ correspond
  to all inner products of $\psi^{i}_{k,m} \in \mcd$ at level
  $i$. Similarly for $w$, where $w = Tu$.}
\end{center}
\label{SRBF:fig1}
\end{figure}
One interesting observation of the matrix $K_{W}$ is that most of the
information of the matrix is contained in a few entries. Indeed, for
integral equations it can be shown that an adapted HB discretization
matrix requires only $\mco(N log(N)^{3.5})$ entries to achieve
optimal asymptotic convergence \cite{Castrillon2003}. This has been
the approach that was followed behind the idea of wavelet
sparsification of integral equations
\cite{beylkin1991,alpert1993a,alpert1993b,Castrillon2003,amaratunga2001,petersdorff1996,petersdorff1997}.

However, it is not necessary to compute the entries of $K_{W}$ for
efficiently inverting the matrix, but instead we {\it only} have to
compute matrix vector products of the submatrices $K^{i,j}_{W}$ in
$\mco(N)$ or $\mco(N log(N))$ computational steps.
\subsection{Preconditioner}
\label{FSEAC}

One key observation of the matrix $K_{W}$ is that each of the blocks
$K^{i,j}_{W}$ is  well conditioned. Our experiments indicate that
this is the case even for non uniform placement of the nodes.
We propose to use two kinds of preconditioners on the decoupled RBF
problem: a block SSOR and a diagonal preconditioner based on the
multi-resolution matrix $K_{W}$. The block SSOR multi-resolution
preconditioner shows better iteration counts and is a novel approach
to preconditioning. However, in practice, the simplicity of the
diagonal preconditioner makes it easier to code and is faster per
iteration count for the size of problems in which we are interested.

The preconditioner on the decoupled RBF takes the form of
the following problem:
\begin{equation}
\bar{P}^{-1}K_{W}w = \bar{P}^{-1}d_{W},
\label{SSOR:eqn1}
\end{equation}
\noindent where $K_{W} \rightarrow L_{W} + D_{W} + L^{H}_{W}$
and
\[
L_{W} = \left[
\begin{array}{cccc}
  0 & 0 &  0 & 0 \\
  K^{1,0}_{W} & 0 & 0 & 0 \\
  \vdots & \ddots & 0 & 0 \\
  K^{n,0}_{W} & \dots & K^{n,n-1}_{W} & 0
\end{array}
\right]\,\,\,\, and \,\,\,\, D_{W} = \left[
\begin{array}{cccc}
  K^{0,0}_{W} & 0 &  0 & 0 \\
  0  & K^{1,1}_{W} & 0 & 0 \\
  0 & 0 & \ddots & 0 \\
  0 & 0 & 0 & K^{n,n}_{W}
\end{array}
\right].
\]
\noindent The block preconditioner is constructed as $\bar{P} = (L_{W} +
D_{W}) D_{W}^{-1} (L^{H}_{W} + D_{W})$.

We can solve this system of equations with a restarted GMRES (or
MINRES since the matrices are symmetric) iteration \cite{saad1986}. To
compute each iteration efficiently we need each of the matrix vector
products of the blocks $K^{i,j}_{W}$ to be computed with a fast
summation method. We have the choice of either computing each block as
matrix-vector products from a fast summation directly, or a sparse
preconditioner that can be built and stored.
\subsubsection{Fast Summation}
\label{subsubsection:fast}
It is not necessary to compute the matrix $K_{W}$ directly, but to
employ approximation methods to compute matrix-vector products $K_{W}
\alpha_{W}$ efficiently. To such end we make the following assumption.
\begin{assumption} Let $\vec{y}_{1}, \vec{y}_{2}, \dots,
  \vec{y}_{N_{1}} \in \mathbb{R}^{3}$, $c_{1}, c_{2}, \dots c_{N_{1}}
  \in \mathbb{R}$, $R_{BF} := span($ $K(x, \vec{y}_{1}),$ $K(x,
  \vec{y}_{2}) \dots, K(x, \vec{y}_{N_1}))$, and $T =
  span\{\tilde{\phi}_{1}, \tilde{\phi}_{2}, \dots, \tilde{\phi}_{q}
  \}$, for some set of linearly independent functions
  $\tilde{\phi}_{1}, \tilde{\phi}_{2}, \dots, \tilde{\phi}_{q}.$ We
  are interested in the evaluation of the RBF map
\[
\phi(\vec x;  \vec y_{1}, \dots, \vec{y}_{N_{1}})
:= \sum_{i=1}^{N_{1}} c_{i} K(\vec{x}, \vec y_{i}),
\]
\noindent where $\vec x \in \mathbb{R}^{3}$. Suppose there exists a
transformation $F(\phi( \vec x;\vec y_{1}, \dots, \vec{y}_{N_{1}})):R_{BF} \rightarrow T$ with $\mco
(N_{1})$ computational and storage cost.  Moreover, any successive
evaluation of $F(\phi(\vec x; \vec y_{1}, \dots, \vec{y}_{N_{1}}))$ can be
 performed on the basis functions of $T$
 in $\mco(1)$
operations and
\[
|F(\phi (\cdot)) - \phi(\cdot)| \leq C_{F}A \left( \frac{1}{a}
 \right)^{\tilde{p} + 1},
\]
\noindent where $\tilde{p} \in \mathbb{Z}^{+}$ is the order of the
fast summation method, $A = \sum_{i=1}^{N_{1}} |c_{i}|$, $C_{F} > 0 $
and $a > 1$.
\label{errorestimates:assumption3}
\end{assumption}
There exist several methods that satisfy, or nearly satisfy,
Assumption \ref{errorestimates:assumption3}. In particular we refer to
those based on multi-pole expansions and the Non-equidistant Fast
Fourier Transform \cite{beatson1997,potts2004,ying2004}.

The system of equations (\ref{SSOR:eqn1}) can now be solved using an
inner and outer iteration procedure.  For the outer loop a GMRES
algorithm is used, where the search vectors are based on the matrix
$\bar{P}^{-1}$ $K_{W}$.

The inner loop consists of computing efficiently the matrix-vector
products $\bar{P}^{-1}K_{W}$ $\alpha_{W}$, for some vector $\alpha_{W}
\in \mathbb{R}^{N}$ . This computation is broken down into two steps:

\noindent {\bf Step One} To compute efficiently $K_{W}\alpha_{W}$ for
each matrix vector product $K_{W}^{i,j}\alpha^{j}_{W}$, we fix
$\psi^{i}_{k,m}$ from Equation (\ref{SRBF:eqn2}) and then transform
the map
\[
\sum_{\psi^{j}_{l,g} \in \bar{D^{j}} } \sum_{\vec{y}_{b} \in X}
K(\vec{x}_{a},\vec{y}_{b})
\psi^{j}_{l,g}[F_{q}(F^{-1}_{p}(\vec{y}_{b}))] \alpha^{j}_{l,g},
\]
\noindent for all the vectors $\psi^{j}_{l,g} \in \bar{D^{j}}$, into a
new basis $\{ \tilde{\phi}_{1}, \tilde{\phi}_{2}, \dots,
\tilde{\phi}_{q}\}$.  The computational cost for this procedure is
$\mco(N_{1})$, where $N_{1}$ corresponds to the number of non-zero
entries of all $\psi^{j}_{l,g} \in \bar{D}^{j}$.  Since the
computational cost for evaluating the new basis on any point $
\vec{x}_{a} \in X$ is $\mco(1)$, then the total cost for
calculating each row of $\bar{K}^{ i , j }_{W} \alpha^{j}_{W}$ is
$\mco( N_{1} + N_{2})$, where $N_{2}$ is equal to all the non-zero
entries of $\psi^{i}_{k,m}$.

Now, since for each $j =0, \dots, n$, $\left| \cup_{k \in \mck(j)}
  B^{j}_{k} \right| \leq N$, then $N_{1}$ is bounded by $CN$ for some $C
> 0$. For the same reason $N_{2}$ is also bounded by $CN$. This
implies that the total cost for evaluating the matrix vector products
$K_{W}\alpha_{W}$ is $\mco$ $((n+1)^{2}N)$.


\noindent {\bf Step Two}: The computation of $\bar{P}^{-1}\beta_{W}$,
where $ \beta_{W} := K_{W}\alpha_{W}$ is broken up into three
stages. First, let $\gamma_{W} := (L_{W}+D_{W})^{-1} \beta_{W}$,
then
\[
\left[
\begin{array}{cccc}
  K^{1,1}_{W}& 0 &  \dots &  0\\
   K^{2,1}_{W} & K^{2,2}_{W} & \ddots & \vdots  \\
  \vdots & 0 & \ddots &  0\\
  K^{n,1}_{W} & \dots & K^{n,n-1}_{W} & K^{n,n}_{W}
\end{array}
\right] \left[
\begin{array}{c}
  \gamma^{1}_{W} \\
  \gamma^{2}_{W} \\
  \vdots \\
  \gamma^{1}_{W}
\end{array}
\right] = \left[
\begin{array}{c}
  \beta^{1}_{W} \\
  \beta^{2}_{W} \\
  \vdots \\
  \beta^{n}_{W}
\end{array}
\right].
\]


\noindent Since $(L_{W}+D_{W})$ has a block triangular from, we
can solve the inverse-matrix vector product with a back substitution
scheme. Suppose that we have found $\gamma^{1}_{W}, \dots,
\gamma^{i-1}_{W}$, then it is easy to see from the triangular
structure of $(L_{W}+D_{W})$ that
\[
\gamma^{i}_{W} = (K^{i,i}_{W})^{-1}[\alpha^{i}_{W} - \sum_{k=1}^{i-1}
K^{i,k}_{W}\gamma^{k}_{W}].
\]
\noindent The cost for evaluating this matrix vector product with a
fast summation method is $\mco ((n+1)^{2} N + k (n+1) N ) $. The
last term comes from the block matrices in $D_{W}$, which are inverted
indirectly with $k$ Conjugate Gradient (CG) iterations
\cite{Hestenes1952,gene1996}. In Section \ref{results} we show
numerical evidence that $k$ converges rapidly for large numbers of
interpolating nodes.

The second matrix vector product, $\eta := D_{W}\gamma_{W}$ is
evaluated in $\mco (N)$ using a fast summation method.  Finally
the last matrix vector product $\mu := (L^{H}_{W}+D_{W})^{-1} \eta_{W}$
can be solved in $\mco ((n+1)^{2} N + k nN )$ by again using a
back substitution scheme.
\begin{remark} For many practical distributions of the interpolating
  nodes in the set $X$, the number of refinement levels $n+1$ is
  bounded by $C_{1}\log{N}$ \cite{beatson1997}. For these types of
  distributions the total cost for evaluating
  $P_{W}^{-1}K_{W}\alpha_{W}$ is $\mco(N log^{2}{N})$ assuming $k$
  is bounded.
\label{hbconstruction:remark10}
\end{remark}
This approach is best for large scale problems where memory becomes an
issue and for large vanishing moments. For small to medium size
problems the blocks $K_{W}^{i,i}$ can be computed in sparse form and
then stored for repeated use.
\subsubsection{Sparse Preconditioners} In this section we show how to
produce two types of sparse preconditioners by leveraging the ability
of HB to produce compact representations of the discrete operator
matrices.

The key idea is to produce a sparsified matrix $\tilde{P}$ of $\bar{P}$
from the entries of the blocks $K^{i,j}_{W}$. This is done by choosing
an appropriate strategy that decides which entries to keep, and which
ones not to compute.

Although it is possible to construct an accurate approximation of
$\bar{P}$ and $K_{W}$ for all the blocks $K^{i,j}_{W}$ $(i,j = 0 \dots
n)$, the computational bottleneck lies in computing the matrix vector
products with $(K^{i,i}_{W})^{-1}$. Thus it is sufficient to compute
the sparse diagonal blocks of $\tilde{K}_{W}$. The off-diagonal blocks
are computed using the fast summation method described in section
\ref{subsubsection:fast}.
\begin{definition} For every vector $\psi^{i}_{k,m} \in \mcd$ and
  the associated support box $B_{i}^{m} \in \mcb$, define the set
  $L^{i}_{m}$ to be the union of $B_{i}^{m}$ and all boxes in
  $\mcb_{j}$ that share a face, edge or corner with $B_{i}^{m}$
  i.e. the set of all adjacent boxes.
\end{definition}
To produce the sparse matrix $\tilde{P}$ we execute the following
strategy: For each entry in $K^{i,i}_{W}$ corresponding to the adapted
HB vectors $\psi^{i}_{k,m}$, $\psi^{i}_{l,g} \in \mcd$, we only
compute this entry if
\begin{equation}
  dist(L^{i}_{k},L^{i}_{l}) := \inf_{\vec x,\vec y} \|L^{i}_{k}(\vec x) -
  L^{i}_{l}(\vec y) \|_{l_{2}(\mathbb{R}^{3})} \leq \tau_{i,i},
\label{SRBF:eqn4}
\end{equation}
\noindent where $\tau_{i,i} \in R^{+}$ for $i = 0 \dots n$. For an
appropriate distance criterion $\tau_{i,i}$ we can produce a highly
sparse matrix $\tilde{K}_{W}^{i,i}$ that is close to $K_{W}^{i,i}$
(and respectively $\tilde{P}$) in a matrix 2-norm sense.
\begin{definition} The distance criterion $\tau_{i,j}$ is
set to
\begin{equation}
\tau_{i,j} :=
2^{n - i},
\label{SRBF:eqn5}
\end{equation}
\end{definition}

With this distance criterion it is now possible to compute a sparse
representation of the diagonal blocks of $K_{W}$. It is not hard to
show that for kernels that satisfy Assumption
\ref{errorestimates:assumption2} the decay of the entries of the
matrix $\tilde{K}_{W}$ is dependent on the distance between the
respective blocks and the number of vanishing moments.  If $p$ is
chosen sufficiently large (for a biharmonic $p = 3$ is sufficient),
the entries of $\tilde{K}_{W}$ decay polynomially fast, which leads to
a good approximation to $K_{W}$.

Under this sparsification strategy, it can be shown that $\| K_{W} -
\tilde{K}_{W} \|_{2}$ decays exponentially fast as a function of the
degree of vanishing moments $p$ with only $\mco(Nn^2)$ entries
in $\tilde{K}_{W}$. The accuracy results have been derived in more detail in
an upcoming paper we are writing for anisotropic spatially varying RBF
interpolation \cite{castrillon2010}.
\begin{lemma} Let ${\mathbb N}(A) :\mathbb{R}^{N \times N} \rightarrow
\mathbb{R}^{+}$, be the number of non-zero entries for the matrix $A$,
then we have
\begin{equation} {\mathbb N}(\tilde{K}^{i,i}_{W}) \leq 8M(p)7^{3}N
\label{eqBndNKWii}
\end{equation}
\label{Complexity:lemma1}
\end{lemma}
\begin{proof} First, identify the box $L^{i}_{k,m}$ that embeds
  $\psi^{i}_{k,m}$ and the distance criterion $\tau_{i,i}$ associated
  with that box.  Now, the number of vectors $\psi^{i}_{l,g}$ and
  corresponding embedding $L^{i}_{l,g}$ that intersect the boundary
  traced by $\tau$ is equal to $(2^{-i}3 + 2 \tau_{i,i} +
  2^{-i+1})^{3} / 2^{-i} \leq 2^{3(i-i)}7^3 = 7^3$ (as shown in Figure
  \ref{SRBF:fig3}).  From Remark \ref{hbconstruction:remark2} there
  are at most $8M(p)$ HB vectors per cube. The result follows.
\end{proof}
To compute the block diagonal entries of $\tilde{K}^{i,i}_{W}$, for $i
= 0, \dots, n$ in $\mco(N \log N )$ computational steps, we employ
a strategy similar to the fast summation strategy in section
\ref{subsubsection:fast}. For each row of $\tilde{K}^{i,i}_{W}$,
locate the corresponding HB $\psi^{i}_{k,m}$ from Equation
(\ref{SRBF:eqn2}) and transform the map
 \begin{equation}
   \sum_{k \in \mck(n)}\sum_{k' \in \mck(n)} \sum_{{e}_{a}
     \in {\bf B}^{n}_{k}}
   K(F_{p}(\vec{x}, e_{a}))
   \psi^{i}_{k,m}[F_{q}(e_{a})]
 \label{sparse:eqn1}
 \end{equation}
 \noindent into an approximation $G(\vec x, \psi^{i}_{k,m}) :=
 \sum_{i=1}^{q} c^{ \psi^{i}_{k,m}}_{i}\tilde{\phi}_{i}$ by applying a
 fast summation method that satisfies Assumption
 \ref{errorestimates:assumption3}. Any entry of the form $a(
 \psi^{i}_{k,m},\psi^{i}_{l,g})$ can be computed by sampling
 $G(\vec{x})$ at locations corresponding to the non-zero entries of
 $\psi^{i}_{l,g}$, and the sampled values can be used to multiply and
 sum through the non-zero values of $\psi^{i}_{l,g}$.
\begin{theorem} Each block  $\tilde{K}^{i,i}_{W}$ is
computed in at most $\mco(N)$ steps.
\end{theorem}
\begin{proof}
  The cost for computing the basis of $G(\vec{x})$ corresponding to
  $\psi^{i}_{k,m}$ is at most $\mco(N^{i}_{k,m})$, where
  $N^{i}_{k,m}$ is the number of non zeros of $\psi^{i}_{k,m}$. Now,
  since $\left| \cup_{k \in \mck(i)} B^{i}_{k} \right| \leq N$ the
  cost of computing $G(\vec{x},\psi^{i}_{k,m})$ for all the vectors
  $\psi^{i}_{k,m}$ at level $i$ is $ \sum_{\psi^{i}_{k,m} \in
    D^{i}_{k}, k \in \mck(i)}$ $N^{i}_{k,m} $ $= \mco(N)$.

  For each row in $\tilde{K}^{i,i}_{W}$, from Lemma
  \ref{Complexity:lemma1} there is at most $8M(p)7^{3}$ entries. This
  implies that for each vector $\psi^{i}_{k,m}$ we need only $\mco(1)$
  evaluations of $G(\vec{x},\psi^{i}_{k,m})$ to compute a row of
  $\tilde{K}^{i,i}_{W}$. Now, if we sum up the cost of evaluating
  $G(\vec{x},\psi^{i}_{k,m})$ for all the rows then the total cost for
  evaluating $\tilde{K}^{i,i}_{W}$ is $\mco(N)$.
\end{proof}
\begin{remark} For each entry in $\tilde{K}^{i,i}_{W}$, the
  corresponding basis vectors $\psi^{i}_{k,m},\psi^{j}_{l,g}$ can be
  found in $\mco(n)$ computational steps. This is easily achieved
  by sorting the set of cubes $\{ B^{j}_{l} \}_{ l \in \mck,
    j=1,\dots,n}$ with an octree structure, i.e. a parent-child
  sorting.
\end{remark}
\begin{remark} Note that further improvements in computation can be
  done by observing that $\psi^{i}_{k,m}$ is a linear combination of
  the vectors $\phi^{i-1}_{l,o} \in V^{i-1}$. Thus equation
  (\ref{sparse:eqn1}) can be written as a linear combination of
 \begin{equation}
   \sum_{k \in \mck(n)}\sum_{k' \in \mck(n)} \sum_{{e}_{a}
     \in {\bf B}^{n}_{k}}
   K(\vec{x},  F_{p}(e_{a}))
   \phi^{i-1}_{l,o}[F_{q}(e_{a})].
 \label{sparse:eqn2}
\end{equation}
\noindent If two vectors $\psi^{i}_{k,m}$ and $\psi^{i}_{k,m'}$ are in
the same cube then it is sufficient to compute equation
(\ref{sparse:eqn1}) once and apply the coefficients computed in the
construction of the entire HB. In addition, if two vectors
$\psi^{i}_{k,m}$ and $\psi^{i}_{k',m'}$ share the same vector
$\phi^{i}_{l,o} \in V^{i-1}$, the same procedure can be applied.  In
our results in Section \ref{results} we apply this scheme to compute
the SSOR and diagonal blocks.
\label{MultiRBF:remark13}
\end{remark}
\begin{figure}
  \psfrag{a}{$L^{i}_{k}$} \psfrag{b}{$L^{j}_{l}$}
  \psfrag{c}{$\psi^{i}_{k,m}$} \psfrag{d}{$\psi^{j}_{l,g}$}
  \psfrag{e}{$\tau_{i,j}$} \psfrag{f}{$vh^{j}$} \psfrag{g}{$h^{i}$}
  \psfrag{h}{$h^{j}$}
\begin{center}
\includegraphics[scale = 1.2]{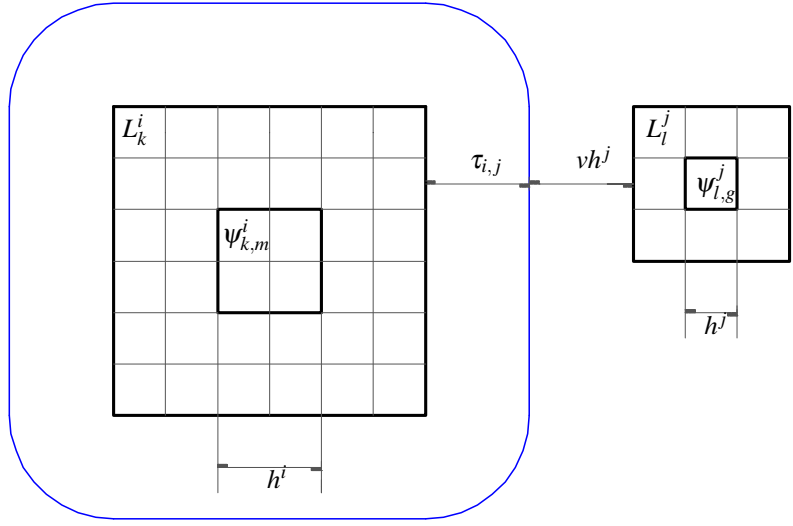}
\end{center}
\caption{Distance criterion cut-off boundary for the cube $L^{i}_{k}$,
corresponding to all the vectors $\psi^{i}_{k,m} \in {\bar
D}^{i}_{k}$.  Assume $j \geq i$ and $h = 2^{-1}$, and each cube
$B^{i}_{k}$ is evenly divided by $B^{j}_{l}$. With this in mind,
the cut-off criterion traces a cube of length $2\tau_{i,j}$ plus the
side length of $L^{j}_{l}$. For any vector $\psi^{j}_{l,g}$ such that
$L^{j}_{l}$ crosses the cut-off boundary, we compute the corresponding
entries in the matrix $\tilde{K}^{i,i}_{W}$.
}
\label{SRBF:fig3}
\end{figure}
\begin{remark}

  As our results show a very simple, but effective, diagonal
  preconditioner can be built from the blocks of $K^{i,i}_{W}$.  In
  particular
\[
P:= diag \left( \left[
\begin{array}{cccc}
K^{1,1}_{W}& 0 &  \dots &  0\\
0 & K^{2,2}_{W} & \ddots & \vdots  \\
0 & 0 & \ddots &  0\\
0 & 0 & 0 & K^{n,n}_{W}
\end{array}
\right] \right).
\]
This preconditioner is also much easier to construct in practice.
\end{remark}
\begin{figure}[htpb]
\begin{center}
\vspace{-3mm}
\psfrag{a}[c]{{\bf Test Case \#1}}
\psfrag{b}[c]{Top view}
\psfrag{c}[c]{Side view}
\psfrag{d}[c]{Side view}
\psfrag{x}[c]{x}
\psfrag{y}[c]{y}
\psfrag{z}[c]{z}
\includegraphics[width=5in, height=3.5in]{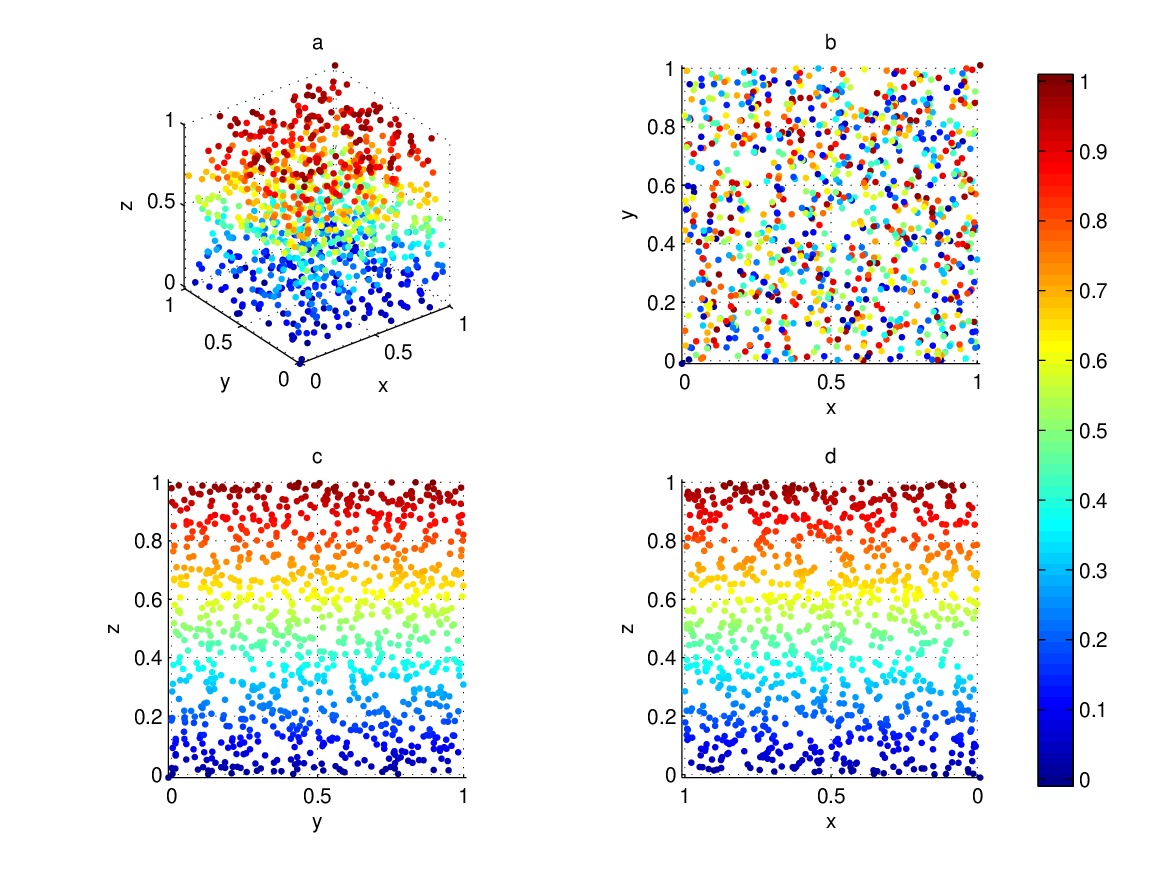}
(a) \\
\psfrag{e}[c]{{\bf Test Case \#2}}
\psfrag{f}[c]{Side View}
\includegraphics[width=4in, height=3.5in]{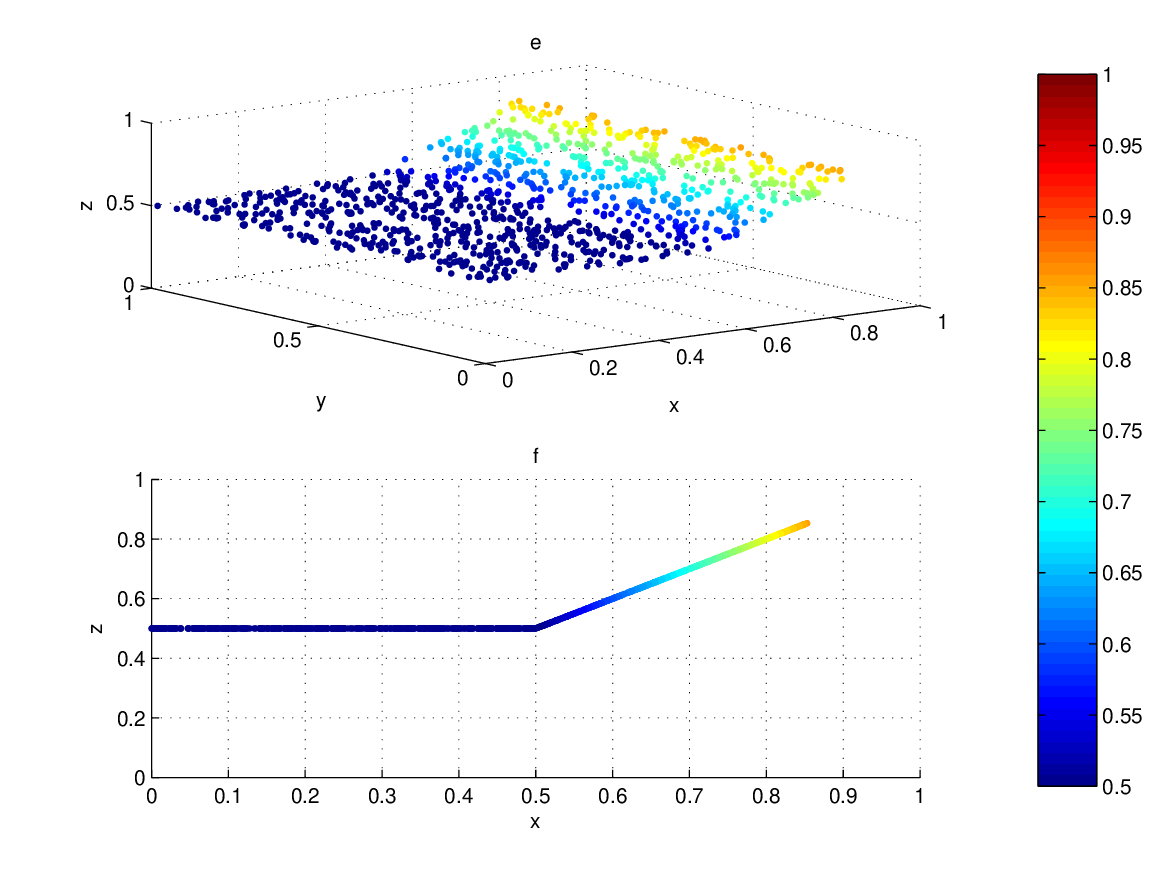}\\
(b)
\end{center}
\vspace{-3mm}
\caption{ {\bf (a) Test Case \#1 Cube RBF interpolating set:}
  Interpolating set with a thousand nodes with
  orthographic views.  The colorbar indicates the height (z-axis) of
  the interpolating nodes. {\bf (b) Test Case \#2 V-plane RBF interpolating
    set},  with one thousand nodes.}
\label{SRBF:fig9}
\end{figure}

\section{Numerical Results}
\label{results} In this section we apply the multi-resolution
method developed in section \ref{RBF:Mutli} to RBF interpolation
problems.  These will be of different sizes and polynomial orders
for the biharmonic, multiquadric and inverse multiquadric function
kernel in $\mathbb{R}^{3}$. These kernels can be written in a
common form $K (r):= (r^2 + \delta^2)^{l/2}$, where $r:=|x|$,
$\delta \in \mathbb{R}$ and $l \in \mathbb{Z}$. The distribution
of the nodes in $X$ are separated into two cases.

{\it Test Case 1:} We test our method on several sets of randomly
generated interpolating nodes in the unit cube in $\mathbb{R}^{3}$
as shown in Figure \ref{SRBF:fig9}. The sets of interpolating
nodes $\{ X_{1}, \dots, X_{e} \} $ vary from 1000 to 512,000
nodes.  Each set of interpolating nodes is a subset of any other
set with bigger cardinality, i.e., $X_{l} \subset X_{l+1}$.  The
function values on each node are also grouped into $e$ sets $\{
b_{1}, \dots, b_{e}\} $ with randomly chosen values and satisfy
also $b_{l} \subset b_{l+1}$.

{\it Test Case 2:} For this second test we apply a projection of the
data nodes generated in {\it Test Case 1} onto two non-orthogonal
planes $\mathbb{R}^{3}$, then remove any $two$ nodes that are less
than $10^{-4}$ distance from each other. The V-plane intersecting are
shown in Figure \ref{SRBF:fig9}. Due to the sharp edges, this test
case is significantly harder than Test Case 1 and the test examples in
\cite{Gumerov2007}. Note, that only about $0.1\%$ of the centers were
eliminated and the number of nodes in the table is approximate.

{\it Test Setup:} The implementation of the multi-resolution discrete
HB method is performed in C++ and compiled with the Intel CC compiler.
The GMRES algorithm is incorporated
from PETSc (Portable, Extensible Toolkit for Scientific computation)
libraries \cite{www:petsc} into our C++ code.  Inner and outer
iterations are solved using a GMRES algorithm with 100-iteration
restart. In the rest of this section when we refer to GMRES
iterations, we imply restarted GMRES with a restart for every 100
iterations.  Since the preconditioned system will introduce errors in
the RBF residual of the original Problem \ref{RBFI:problem}, the
accuracy of the GMRES is adjusted such that the residual $\epsilon$ of
the unpreconditioned RBF system is less than $10^{-3}$. In Tables
\ref{SRBF:table2} and \ref{SRBF:table5} the GMRES accuracy residual
are reported.

All the numerical tests with a fast summation method are performed
with a single processor version of the Kernel-Independent Fast
Multipole Method (KIFMM) 3D code (http://mrl.nyu.edu/ $\sim$harper/kifmm3d/documentation/index.html).
 This code implements the
algorithm described in \cite{ying2004}.  The accuracy is set to
relative medium accuracy ($10^{-6}$ to $10^{-8}$). In addition, all
numerical timings presented in this paper are wall clock times.

The C++ code was also compiled  for a single core on the Dell Precision
T7500 workstation with Linux Ubuntu 11.04, 12 core Xeon X5650 at 2.67
GHZ with 12 MB Cache.  All results (except for the Condition number test)
where performed sequentially on
a single core of the same processor. Similar results where also
observed with a single thread of Core i7 1.66 GHZ processor and on a
single core of an Intel(R) Core(TM)2 Quad CPU Q9450 @2.66GHZ processor
(12 MB L2 Cache).  At some point we will make available the code for
the public with instructions for compilation.

{\bf Test Examples:}

{{\it Condition number $\kappa$ of underlying system of equations with
    respect to scaling all the domain.}  One immediate advantage our
  method has over a direct method is the invariance of the
  conditioning of the system of equations with respect to the scale of
  the polynomial domain. This is a consequence of the construction of
  the HB polynomial orthogonal basis.

Removing the polynomial source of ill-conditioning makes the system
easier to solve. The condition number of the full RBF interpolation
matrix is sensitive to the scaling of the domain. We show this by
scaling the domain by a constant $\alpha \in \mathbb{R}$.

As shown in Table \ref{SRBF:table1} the condition number for the 1000
center problem with $m=3$ and $K(r)=r$ deteriorates quite rapidly
with scale $\alpha$. In particular, for a scaling of 1000 or larger (
0.01 or smaller) an iterative method, such as GMRES or CG, stagnates.
We note that the invariance of the condition number of the decoupled
system was also observed in \cite{beatson2000,sibson1991}.

Another important observation is that the same result will apply for a
multiquadric, or inverse multiquadric of the form $K(r) = (r^2 +
\delta^2)^{\pm l}, \delta \in \mathbb{R}$, due to the polynomial
decoupling from the RBF matrix. In general, this will be true for any
strictly conditionally positive (or negative) definite RBF.  However,
the matrix $K_{W}$ will still be subject to the underlying condition
number of $K$. In other words, if $\kappa(K)$ deteriorates
significantly with scale then $K_{W}$ will also be ill-conditioned.}

\begin {table}[htbp]
\begin{center}
  \begin {tabular} {| c| c | c | c | c | c |}
    \hline Scale & $\alpha = 0.01 $ & $\alpha = 0.1$ & $\alpha = 1$ &
    $\alpha = 100$ & $\alpha = 1000$ \\ \hline \hline $\kappa$ of
    RBF system & $ 4.5\times10^{24}$ & $5.7\times10^{13}$ & $5.2\times10^{6}$
    & $7.9\times10^{11}$ & $5.3\times10^{17}$ \\ \hline
    $\kappa( K_{W} )$ & 762 & 762 &  762 &  762 & 762 \\
    \hline
\end {tabular}
\\
\end{center}
\caption{Condition number for RBF system matrix (equation
  \ref{RBFI:problem}) versus scale of the problem for a thousand nodes
  for Test Case 1 with respect to the biharmonic $K(r) := r$. As
  observed, increasing the scale by alpha the condition number
  deteriorates very rapidly.  In particular, for a condition number
  higher than the reciprocal of machine position and the GMRES or CG
  algorithm stagnates.}
\label {SRBF:table1}
\end {table}

{\it Biharmonic RBF, m = 3 (cubic) and p = 3} This is an example
of a higher order polynomial RBF interpolation. We test both the
SSOR and diagonal preconditioner on Test Case 1 \& 2.  For the
preconditioner the accuracy of the GMRES outer iterations is set
such that the residual $\epsilon:=\| K_{W}w - d_{W} \| \leq
10^{-3}$. 
Due to the condition number of the blocks they quickly converge with
either a CG, or a GMRES solver. Moreover, the number of iterations
appear to grow slowly with size.

In Table \ref{SRBF:table2} (a) the iteration and timing results for
the sparse SSOR preconditioner for Test Case 1 \& 2 are shown.  For
Test Case 1, the number of restarted GMRES iterations grows as
$\mco(N^{0.55})$. Fitting a linear regression function to the log-log
plot leads to a growth of $\mco(N^{1.85})$ for time complexity. Test
Case 2 (v-plane) is a harder problem due to the corner and the
projection of the random data from Test Case 1 onto two planes at 135
degrees to each other. The total GMRES iteration grows as
$\mco(N^{0.54})$ at time complexity $\mco(N^{1.85})$.

In Table \ref{SRBF:table2} (b) the iteration and timing results for
the diagonal preconditioner with Test Case 1 \& 2 are shown.  We can
observe that although the GMRES iteration count is higher than that of
the SSOR preconditioner ($CN^{0.51}$), the simplicity of the
preconditioner allows every matrix-vector product to be computed much
faster.  Fitting a line to the log data leads to a total time
complexity increases of $\mco(N^{1.6})$. The memory constraints are
also much lower than the SSOR since only $N$ entries are needed to be
stored for the preconditioner. For Test Case 2 (v-plane), the increase
of time complexity ($\mco(N^{1.7})$ and the GMRES iteration count
$\mco(N^{0.73})$ reflects that it is a harder problem than Test Case
1. Another observation is that the time required to compute the
diagonal preconditioner is about one half compared to Test Case
1. This is due to the adaptive way we compute the diagonal, recall
Remark \ref{MultiRBF:remark13}.

\begin {table}[htbp]

\begin{center}
  \begin {tabular} {| c || c | r | r | r || c | r | r | r |}
  \hline
  &  \multicolumn{4}{|c||}{Test Case 1}  &
   \multicolumn{4}{|c|}{Test Case 2}  \\
  \cline{2-9}
  N &  GMRES & \,\,$\tilde{K}^{i,i}_{W}(s)$\,\, &
  Itr (s) & Total (s)  & GMRES & \,\,$\tilde{K}^{i,i}_{W}(s)$\,\, &
  Itr (s) & Total (s)  \\
  \hline
    1000   & 10  &   1  & 8     &    9  &   22 &    1 &    16  &     17 \\
    2000   & 15  &   2  & 13    &   15  &   29 &    2 &    41  &     43 \\
    4000   & 21  &   6  & 39    &   45  &   43 &   5  &   178  &    183 \\
    8000   & 29  &  61  & 275   &   336 &   66 &   44 &   623  &    667 \\
    16000  & 48  & 194  & 798   &  993  &   91 &  118 &  2298  &   2416 \\
    32000  & 71  & 815  & 3907  &  4722 &  118 &  612 &  12288 &  12900 \\
    64000  & 99  & 1754 & 13841 & 15595 &  195 &  951 &  28500 &  29451 \\
    128000 & 134 & 5547 & 29165 & 34712 &  305 & 2572 &  98505 & 101078 \\
    \hline
\end {tabular}
\\
\vspace{1mm} (a) SSOR Preconditioner Test Case 1 \& 2\vspace{3mm}

\begin {tabular} {| c || c | r | r| r || c |r |r |r|}
\hline
  &  \multicolumn{4}{|c||}{Test Case 1}  &
   \multicolumn{4}{|c|}{Test Case 2}  \\
  \cline{2-9}
  N & GMRES & Diag. (s)  & Itr (s) & Total (s) & GMRES & Diag. (s)  & Itr (s) & Total (s) \\
  \hline
  1000   & 33  & 1  &      5  &     6  & 90      &  1   & 10   &    11   \\
  2000   & 45  & 2  &       6 &      8 & 102     & 2   & 11 &    13   \\
  4000   & 66  & 6  &      16 &    22  & 147     & 5   & 30 &   34  \\
  8000   & 87  & 62 &      56 &    117    & 269     & 44 &  121   &   165 \\
  16000  & 128 & 195   &        148 &    344   & 355      & 118   &  286  &   404 \\
  32000  & 184 & 813   &        749 &   1563    & 876     & 569   &  2924 &  3493 \\
  64000  & 281 & 1752  &       1817 &   3569   & 1242    &  951    & 4135 &  5087 \\
  128000 & 385 & 5555  &       3949 &  9505    & 3033   &   2573   & 21917& 24491 \\
  256000 & 573 & 14350 &      12130 &  26480   &- &- &- &- \\
  512000 & 769 & 47082 &      44309 & 91391     &- &- & -&- \\
  \hline
\end {tabular}

\vspace{1mm} (b) Diagonal Preconditioner Test Case 1 \& 2
\vspace{3mm}

\end{center}
\caption{Wall clock times results for biharmonic $K(r)=r$, $m = 3$
  (Cubic), $p = 3$.  (a) Iteration and timing results for the sparse
  SSOR preconditioner for Test Case 1 (uniform cube) \& 2 (v-plane).
  The first column is the number of interpolating points. The second
  column is the number of iterations such that $\epsilon$, the
  residual error for the unpreconditioned system, is less than
  $10^{-3}$.  The third column is the time (in seconds) to compute the
  sparse inner blocks $K^{i,i}_{W}$ and the fourth is the time for
  GMRES iterations.  The fifth column is the total time (in seconds)
  for solving the RBF problem.  The remaining columns are for Test
  Case 2 and follow the same order as results for Test Case 1.  (b)
  Iteration and timing results for diagonal preconditioner for Test
  case 1 \& 2. The columns are in the same order as before, except
  that the third column and seventh columns are the time involved in
  computing the diagonal preconditioner. }
\label{SRBF:table2}
\end {table}

\begin {table}[htbp]
\begin{center}
\begin {tabular} {| c||c | r | r | r || c | c | r|r | r |}
\hline
  &  \multicolumn{4}{|c||}{ (a) Test Case 1, Multiquadric}  &
   \multicolumn{4}{|c|}{ (b) Test Case 1, Inverse Multiquadric}  \\
  \cline{2-9}
  N &  GMRES & Diag.(s) &Itr (s) & Total (s)  &  GMRES & Diag.(s) & Itr (s) & Total (s)  \\
  \hline
  1000   & 38    &  1  &   1     &     2 &      7   &   1   &    1   &  1   \\
  2000   & 55    &  2  &      5  &      7 &   8       & 3 &     1  & 4   \\
  4000   & 86    &  6  &     13 &    18 &   14      &   8   &      4 &  11  \\
  8000   & 128   &  32 &     41&   73  &    17      &  45   &   9   & 54  \\
  16000  & 195   &  99    &      155&      254  & 27       &    138   &    28 &  166 \\
  32000  &  362  & 233    &     486&      720  &   63      &  343     &   119   & 462  \\
  64000  & 684   &  757   &   2217&    2975  &      84     &  1131    &       414     &  1546\\
  128000 & 1059  & 2357   &  7637&   9994  &       112     &  3494    &      985     &  4480\\
  \hline
\end{tabular}
\end{center}
\caption{Iteration and timing results for diagonal preconditioner,
  multiquadric $K(r):=(r^2 + 0.01^2)^{ \pm \frac{1}{2}}$, and test
  case 1 (uniform cube), $m = 3$, $p = 3$  for (a) Multiquadric (+1/2)
  and (b) Inverse multiquadric (-1/2).}
\label {SRBF:table5}
\end{table}

{\it Multiquadric and inverse multiquadric RBF,
  $m = 3$ (cubic), $p = 3$.} For the case
of the multiquadrics with $\delta = 0.01$, the iteration count increases
significantly, as shown in Table \ref{SRBF:table5}(a). The number of
GMRES iterations increases as $CN^{0.7}$.  This is a harder problem
to solve due to the ill-conditioning introduced by the constant term
$\delta$, as reflected by the increase in the number of GMRES
iterations. Fitting a line through the log-log plot of the total time
leads to a $CN^{1.8}$ time complexity.

In contrast, the inverse multiquadrics result shown in Table
\ref{SRBF:table5}(b) is a better conditioned problem leading to around
the same complexity as for the biharmonic case, but the constant is
lower.  We note that to achieve comparable interpolation accuracy, the
value of $\delta$ for the inverse multiquadric generally needs to be larger
than for the multiquadric case. And the larger the $\delta$ the more
ill-conditioned the RBF interpolation problem.

\section{Conclusions}
\label{conclusions}

In this paper we construct a class of discrete HB that are adapted
both to the RBF kernel function and the location of the
interpolating nodes. The adapted basis has two main advantages:
First the RBF problem is decoupled, thus solving the scale
dependence between the polynomial and RBF interpolation. Second
with a block SSOR scheme, or a simple diagonal matrix built from
the multi-resolution matrix $K_{W}$, an effective preconditioner
is built that reduces significantly the iteration count. Our
result shows a promising approach for many RBF interpolation
problems.

Further areas of interest as future work:

\begin{itemize}

\item {\it Sparsification of $K_{W}$ matrix.} Due to orthogonality
  properties of the discrete HB a sparse representation
  $\tilde{K}_{W}$ of $K_{W}$ can be constructed where $\| K_{W} -
  \tilde{K}_{W} \|$ is small. The sparse representation is used at
  each iteration in lieu of the dense matrix, thus opening the
  possibility of significantly increasing the time efficiency of each
  matrix vector product.

\item {\it High Dimensional RBF Problems.}  In principle the method
  that we have developed can be extended to high dimensional RBF
  problems.

\item{ {\it Spatially varying anisotropic kernels}.  An interesting
    observation is that the adapted discrete HB leads to a sparse
    multi-resolution RBF matrix representation for spatially varying
    kernels. This type of RBF interpolation has been gaining some
    interest lately due to the ability to better steer each local RBF
    function to increase accuracy.  Due to the spatially varying
    kernel, we cannot use a fast summation method to optimally compute
    each matrix vector product. However, preliminary results show that
    we can sparsify the RBF matrix while retaining high accuracy of
    the solution. Full error bounds and numerical results will be
    described in a following paper that we are currently writing.}

\end{itemize}




\section*{Acknowledgments}
We are grateful to Lexing Ying for providing a single processor version of the
KIFMM3d code.  We also appreciate the discussions, assistance and
feedback from Raul Tempone, Robert Van De Gein, Vinay Siddavanahalli
and the members of the Computational Visualization Center (Institute
for Computational Engineering and Sciences) at the University of
Texas at Austin. In addition, we appreciate the invaluable feedback
from the reviewers of this paper.


\bibliographystyle{amsplain}

\input{HBRBF-Arxiv.bbl}


\end{document}

%% file: HBRBF-Arxiv.bbl
\newcommand{\noopsort}[1]{}
\providecommand{\bysame}{\leavevmode\hbox to3em{\hrulefill}\thinspace}
\providecommand{\MR}{\relax\ifhmode\unskip\space\fi MR }
\providecommand{\MRhref}[2]{%
  \href{http://www.ams.org/mathscinet-getitem?mr=#1}{#2}
}
\providecommand{\href}[2]{#2}